\documentclass[10pt]{article}

\usepackage{dcolumn}
 \newcolumntype{d}[1]{D{:=}{\cdot}{#1} }

\usepackage[colorlinks=true, pdfstartview=FitV, linkcolor=blue,
  citecolor=blue, urlcolor=blue]{hyperref}

\usepackage{amssymb,amsmath, amscd,amsthm,comment}
\usepackage{graphicx}
\usepackage[all]{xy}
\usepackage{xurl}

\newtheorem{theoremx}{Theorem}
\newtheorem{theorem}{Theorem}[section]

\newtheorem{lemma}[theorem]{Lemma}

\newtheorem{proposition}[theorem]{Proposition}
\newenvironment{remark}{\medskip\refstepcounter{theorem}\noindent\textbf{Remark \thetheorem.}}{\medskip}

\def\A{{\mathbb A}}

\DeclareMathOperator{\ord}{ord}

\DeclareMathOperator{\SO}{SO}
\DeclareMathOperator{\PSO}{PSO}

\DeclareMathOperator{\SL}{SL}

\DeclareMathOperator{\GL}{GL}
\DeclareMathOperator{\PGL}{PGL}

\DeclareMathOperator{\Spec}{Spec}
\DeclareMathOperator{\Sel}{Sel}

\DeclareMathOperator{\Pic}{Pic}

\DeclareMathOperator{\Sym}{Sym}
\DeclareMathOperator{\Span}{Span}
\DeclareMathOperator{\Stab}{Stab}

\def\dim{{\rm dim}}

\def\gen{{\rm gen}}

\def\inv{{\rm inv}}

\def\Vol{{\rm Vol}}
\def\R{{\mathbb R}}
\def\F{{\mathbb F}}

\def\FF{{\mathcal F}}

\def\Q{{\mathbb Q}}

\def\J{{\mathcal J}}
\def\C{{\mathcal C}}

\def\Z{{\mathbb Z}}
\def\P{{\mathbb P}}
\def\F{{\mathbb F}}
\def\FF{{\mathcal F}}
\def\Q{{\mathbb Q}}
\def\C{{\mathbb C}}
\def\A{{\mathbb A}}

\def\L{{\mathcal L}}

\def\CC{{\mathcal C}}

\def\max{{\rm max}}

\def\nicerep{{Representation}}
\def\localCondition{{Local Weights}}

\def\cO{{\mathcal O}}
\def\cl{{\rm cl}}
\def\fP{{\mathfrak p}}
\def\fQ{{\mathfrak q}}
\def\sc{{\mathcal C}}
\def\bG{{\mathbb G}}

\DeclareMathOperator{\proj}{proj}
\DeclareMathOperator{\diag}{diag}

\title{Geometry-of-numbers methods over global fields II: \\Coregular representations}
\author{Manjul Bhargava, Arul Shankar, and Xiaoheng Wang}
\date{}

\begin{document}
\maketitle
\begin{abstract}We develop geometry-of-numbers methods to count orbits
  in coregular vector spaces having bounded invariants over any
  global field.   We apply these techniques to bound the average ranks and determine average Selmer group sizes of elliptic curves and Jacobians of hyperelliptic curves over any base global field $F$ of characteristic not $2$, $3$, or $5$.
\end{abstract}

\section{Introduction}\label{sec:intro}

In \cite{BSWfield}, the first of this series of two articles, we extended certain counting methods from the geometry of numbers so that
they may be applied over an arbitrary global number or function field for the purpose of counting integral
orbits in a prehomogeneous vector space. As applications, we obtained natural generalizations (from $\Q$ to general global fields) of results on the density of discriminants of number field extensions of small
degree \cite{DH,dodqf,dodpf}, as well as on cases of the Cohen--Lenstra--Martinet
heuristics for class groups \cite{BV2, BV}.

In recent years, there have also been many works counting orbits over $\Q$ in coregular spaces, with applications to the arithmetic statistics of elliptic curves, hyperelliptic curves, squarefree sieves, and more; see, for example, \cite{B}, \cite{BG1}-\cite{bsddensity}, \cite{Laga}, \cite{LR}, \cite{PS2}, \cite{RT}, \cite{ANS}, \cite{Th}, among many others.
In this article, we develop methods for counting integral orbits of coregular representations over general global fields, thus, in particular, generalizing results from \cite{BS2, BS3, BS4, BS5, B, BG1, SW, BGW}.

Let $F$ denote a fixed global field. That is, $F$ is either a number field or the field of rational functions of a smooth projective and geometrically connected algebraic curve $\sc$ over $\F_q$. When $F$ is a number field, let $\cO$ denote the ring of integers of $F$ and let $M_\infty$ denote the set of archimedean places of $F$. When $F$ is the field of rational functions of a smooth projective and geometrically connected algebraic curve $\sc$ over $\F_q$, we fix some nonempty finite set $D$ of closed points on $\CC$. Let $\cO$ denote the ring of regular functions on $\sc-D$ and let $M_\infty$ denote the finite set of places corresponding to $D$.

When the characteristic of $F$ is not $2$ or $3$, an elliptic curve $E/F$ may be expressed in the form $$E = E_{A,B}:y^2 = x^3 + Ax + B,$$ with $A,B\in F$. Two such elliptic curves $E_{A,B}$ and $E_{A',B'}$ over $F$ are equivalent if and only if there is some $\alpha\in F$ such that $A'=\alpha^4A$, $B'=\alpha^6B$. Hence it is natural to view $(A,B)$ as an element of the weighted projective space $\P(4,6).$ We define the height of $E_{A,B}$, and the height of $(A,B)$, by the usual height on weighted projective space as follows: let $I$ be the fractional ideal
$$I = \{a\in F\colon a^4A\in \cO,a^6B\in\cO\};$$
then
\begin{equation}
H(A,B) := (NI)\prod_{\fP\in M_\infty} \max(|A|_\fP^{1/4},|B|_\fP^{1/6}).
\end{equation}
The product formula implies that this height is well-defined. There is a Schanuel-type count of the number of elements in weighted projective space with bounded height (\cite{AWD}). We note that this showcases one major difference between having multiple invariants (e.g., coefficients of the Weierstrass equation) and having a unique invariant (e.g. the discriminant of a field extension).

We prove the following bound on the average rank of elliptic curves over a global field $F$.

\begin{theoremx}\label{thm:ellrank}
Let $F$ be any global field of characteristic not $2$, $3$ or $5$.  When all elliptic curves $E$ over $F$ are ordered by height, their average rank is at most $1.05$.
\end{theoremx}

\noindent From the works of Goldfeld \cite{Go} and Katz-Sarnak \cite{KS}, it is predicted that this average rank is $1/2$. When $F=\Q$, an improved version of Theorem \ref{thm:ellrank} was proved in \cite{BS5}, with an upper bound of 0.885. The boundedness of the average rank of elliptic curves over $F$, when $F=\F_q(t)$, was proved by de Jong \cite{dJ}, who demonstrated a bound of $4/3+o_q(1)$. Independently, a rank bound of $3/2+o_q(1)$ was proved by Ho--Le Hung--Ngo \cite{Ngo} over function fields over $\F_q$. The average rank of elliptic curves over number fields was proved to be bounded by $3/2$ in \cite{ST}. Theorem~\ref{thm:ellrank} significantly improves known results when $F\neq\Q$ is a number field or a function field of characteristic not dividing $30$.

We consider next hyperelliptic curves. Fix a positive integer $m\geq1$ and let $F$ be a global field of characteristic not $2$. A \emph{monic hyperelliptic curve} over $F$ of degree $m$ is a curve over $F$ given by an affine equation of the form
\begin{equation}\label{eq:monichyp}
C_{c_1,\ldots,c_{m}}=C_f\colon  y^2 = f(x)=x^{m} + c_1x^{m-1} + \cdots +c_{m},
\end{equation}
where $c_1,\ldots,c_{m}\in F$ and the discriminant $\Delta(f)\neq0$.

Two monic hyperelliptic curves $C_{c_1,\ldots,c_m}$ and $C_{c'_1,\ldots,c'_m}$ over $F$ are said to be equivalent if and only if there is some constant $\alpha\in F$ such that $c_i=\alpha^{2i}c'_i$ for all $i=1,\ldots,m$. Hence it is natural to view $(c_1,\ldots,c_m)$ as an element of the weighted projective space $\P(2,4,\ldots,2m).$ We define the height of $C_f$ similarly to the above: let $I$ be the fractional ideal
$$I = \{a\in F:a^{2i}c_i\in \cO,\forall i=1,\ldots,m\}.$$
Then
\begin{equation}
H(c_1,\ldots,c_m) = (NI)\prod_{\fP\in M_\infty} \max(|c_1|_\fP^{1/2},\ldots,|c_m|_\fP^{1/(2m)}).
\end{equation}
Note that it is possible for two inequivalent monic hyperelliptic curves to be isomorphic over $F$. We prove the following bound on the average rank of the Jacobians of monic hyperelliptic curves.

\begin{theoremx}\label{hyperecthm}
Fix a global field $F$ of characteristic not $2$, and a positive integer $m$. Then when all monic hyperelliptic curves over $F$ of degree $m$, up to equivalence, are ordered by height, the average rank of their Jacobians is bounded above by~$3/2$ if $m$ is odd, and $5/2$ if $m$ is even.
\end{theoremx}

\noindent Theorem \ref{hyperecthm}, when $F=\Q$, was proved in \cite{BG1,SW}.

Finally, we consider the family of general (even degree) hyperelliptic curves. Fix a genus $n\geq 1$ and let $F$ be a global field of characteristic not $2$. We consider hyperelliptic curves over $F$ as curves in weighted projective space $\P(1,1,n+1)$ and expressed by an equation of the form
\begin{equation}\label{hypereq}
C_{c_0,\ldots,c_{2n+2}}=C_f\colon  z^2 = f(x,y) = c_0x^{2n+2}+c_1x^{2n+1}y+\cdots+c_{2n+2}y^{2n+2},
\end{equation}
where $c_0,\ldots,c_{2n+2}\in F$. Two curves $C_{c_0,\ldots,c_{2n+2}}$ and $C_{c'_0,\ldots,c'_{2n+2}}$ are said to be equivalent if and only if there is some constant $\alpha\in F$ such that $c_i=\alpha^{2}c'_i$ for all $i=2,\ldots,2n+2$. Hence it is natural to view $(c_0,\ldots,c_{2n+2})$ as an element of the weighted projective space $\P(2,2,\ldots,2).$ We define the height of $C_f$, and the height of $(c_0,\ldots,c_{2n+2})$, as follows: let $I$ be the ideal
$$I = \{a\in F\colon a^{2}c_i\in \cO,\forall i=0,\ldots,2n+2\}.$$
Then
\begin{equation}
H(c_0,\ldots,c_{2n+2}) = (NI)\prod_{\fP\in M_\infty} \max(|c_0|_\fP^{1/2},\ldots,|c_{2n+2}|_\fP^{1/2}).
\end{equation}

We say that a variety over $F$ is {\it locally soluble} if it has a point over $F_\fP$ for every place $\fP$ of~$F$, and is {\it soluble} if it has a point over $F$. We prove the following result on the rarity of rational points on hyperelliptic curves.

\begin{theoremx}\label{thm:generalhyperelliptic}
 Fix an integer $n\geq1$ and a global field $F$ of characteristic not $2$. When all locally soluble hyperelliptic curves over $F$ of genus $n$, up to equivalence, are ordered by height, a positive proportion of them have no points over any odd degree extension of $F$. Moreover, the proportion of locally soluble hyperelliptic curves which have no points over $F$ approaches $100\%$ as $n$ goes to infinity.
\end{theoremx}

\noindent Theorem \ref{thm:generalhyperelliptic} when $F=\Q$ was proved in \cite{B,BGW}.

\medskip

We prove these results on rank bounds and on rational points by studying related Selmer groups. Recall that the $m$-Selmer group $\Sel_m(A)$ of an abelian variety $A$ over a global field $F$ contains a subgroup isomorphic to $A(F)/mA(F)$. Hence a bound on the average size of the $m$-Selmer groups of $A$ gives a bound on the average rank of $A$. We prove the following results on the average sizes of Selmer groups.

\begin{theoremx}\label{ecthm}
Let $n\in\{2,3,4,5\}$. Let $F$ be any global field of characteristic not dividing $6$ when $n\neq 5$, and characteristic not dividing $30$ when $n=5$.  When all elliptic curves $E$ over $F$ are ordered by height, the average size of the $n$-Selmer groups of $E$ is equal to $\sigma(n)$, the sum of divisors of $n$.
\end{theoremx}
\noindent Theorem \ref{ecthm} was proved in \cite{BS2,BS3,BS4,BS5} when $F=\Q$.
An upper bound of $3$ on the average size of the $2$-Selmer groups was proved in \cite{ST} when $F$ is a number field. Upper bounds on the average sizes of the $2$-Selmer and $3$-Selmer groups were proved in \cite{Ngo} and \cite{dJ}, respectively, when $F$ is a function field.

\medskip

Theorems \ref{hyperecthm} and \ref{thm:generalhyperelliptic} regarding rational points on and ranks of Jacobians of hyperelliptic curves are proved via results on $2$-Selmer groups and $2$-Selmer sets of Jacobians of hyperelliptic curves. We prove the following theorem.

\begin{theoremx}\label{hyperecthmsel}
Fix a positive integer $m$ and a global field $F$ of characteristic not $2$. Then when all monic hyperelliptic curves over $F$ of degree $m$, up to equivalence, are ordered by height, the average size of the $2$-Selmer groups of their Jacobians is bounded above by~$3$ if $m$ is odd and by $6$ if $m$ is even.
\end{theoremx}
\noindent Theorem \ref{hyperecthmsel}, when $F=\Q$, was proved in \cite{BG1,SW}.

For a hyperelliptic curve $C$ over $F$, let $J$ denote its Jacobian and let $J^1$ denote the principal homogeneous space for $J$ whose points correspond to the divisor classes of degree one on $C$. When a hyperelliptic curve $C$ is locally soluble, $C$ has no points over any odd degree extension of $F$ if and only if $J^1$ has no point over $F$. A \emph{two-cover} of $J^1$ is an unramified cover  $\pi: Y \rightarrow J^1$ by a principal homogeneous space $Y$ of $J$ such that for all $y \in Y$ and $a \in J$, we have $\pi(y + a) = \pi(y) + 2a \in J^1$. We define the $2$\emph{-Selmer set} $\Sel_2(J^1)$ of $J^1$ as the set of isomorphism classes of locally soluble two-covers of $J^1$. Any rational point $e$ on $J^1(F)$ gives rise to a locally soluble (in fact soluble) two-cover of $J^1$; namely, after identifying $J^1$ with $J$ by subtracting $e$, we pull back the multiplication-by-$2$ map on $J$. Therefore, to prove the statement on odd degree points of Theorem \ref{thm:generalhyperelliptic}, it suffices to prove the following stronger statement.

\begin{theoremx}\label{thm:GHsel2}
 Fix an integer $n\geq1$ and a number field $F$. When all locally soluble hyperelliptic curves $C$ over $F$ of genus $n$, up to equivalence, are ordered by height, a positive proportion have empty $\Sel_2(J^1)$.
\end{theoremx}

To prove this theorem, we compute the average size of $\Sel_2(J^1)$ over any large family of locally soluble hyperelliptic curves over $F$.

\begin{theoremx}\label{thm:GHsel2count}
 Fix an integer $n\geq1$ and a global field $F$ of characteristic not $2$. When all locally soluble hyperelliptic curves $C$ over $F$ of genus $n$, up to equivalence, are ordered by height, the average size of $\Sel_2(J^1)$ is bounded above by $2$.
\end{theoremx}

\noindent Theorems \ref{thm:GHsel2} and \ref{thm:GHsel2count}, when $F=\Q$, were proved in \cite{BGW}. As in \cite{BGW}, Theorem \ref{thm:GHsel2} follows by combining Theorem~\ref{thm:GHsel2count} with a result of Dokchitser--Dokchitser \cite[Theorem A.2]{BGW}. 
The function field analogue of this result is the only missing ingredient to proving Theorem \ref{thm:GHsel2} for function fields with characteristic not $2$.

\medskip

The proofs of our main results are obtained by suitably generalizing the techniques developed in the aforementioned works from $\Q$ to global fields. Specifically, in each case, we consider a coregular representation $V$ of a reductive group $G$ such that the space of invariants corresponds to the family of curves in question. The arithmetic invariant theory of these representations implies that there is a bijection between elements in the desired Selmer groups/sets and certain $G(F)$-orbits, called {\it locally soluble} orbits, in $V(F)$ defined by local conditions at every place of $F$. We partition the relevant orbits into two sets: non-generic ones that correspond to Selmer elements of lower order; and generic ones. The goal then is to apply geometry-of-numbers and sieve techniques over a global field to count the generic orbits, and evaluate the final asymptotic constants using a mass formula. 

The main technical results of this paper are in two parts. First in \S\ref{sec:general}, we give a list of ``axioms'' on the arithmetic of $(G,V)$, which when satisfied allow us to carry out each of the three above steps, namely, the geometry-of-numbers count, the sieve, and the mass formula. This is the essence (and generalization to the global field case) of the strategy used in the literature for the $F=\Q$ case. In the literature, these axioms have been verified over $\Q$ using inputs from algebraic and analytic number theory, combinatorics, algebraic geometry, and representation theory. Our second set of main technical results, established in \S\ref{sec:count}, prove that once these inputs over $\Q$ have been established, the axioms in question can in fact then be verified (with very limited extra effort) over general global fields. In fact, our main results, Theorems \ref{thm:ellrank} through \ref{thm:GHsel2count}, which are proved for $F=\Q$ in \cite{BS2,BS3,BS4,BS5,BG1,SW,B,BGW}, are all established here for general global fields $F$ in \S\ref{sec:elliptic} and \S\ref{sec:hyper}, using our main technical results.

We now describe the general strategy for counting $F$-rational orbits in more detail. First we reduce to counting integral orbits. Because the class group of $G$ over $F$ is not necessarily trivial, it is not true that every locally soluble $G(F)$-orbit with integral invariants contains an element of $V(\cO)$. For every element $\beta$ in the class group of $G$ over $F$, we construct a lattice $V_\beta$ commensurable with $V(\cO)$ and a natural subgroup $G_\beta$ of $G(F)$ commensurable with $G(\cO)$ that acts on $V_\beta$. We show that every locally soluble $G(F)$-orbit in $V(F)$ contains an element of $V_\beta$ for some $\beta$ and so we turn to counting $G_\beta$-orbits on $V_\beta$. Abusing notation, we still call these integral orbits. We note that a locally soluble $G(F)$-orbit in $V(F)$ may contain an element of $V_\beta$ for multiple $\beta$, and an extra weight function is needed to remedy this extra complication.

Embed $F$ into $F_\infty=\prod_{\fP\in M_\infty}F_\fP$ and $V(F)$ into $V(F_\infty)$. We construct (a cover of) a fundamental domain $\FF$ for the action of $G_\beta$ on the subset $V(F_\infty)^{\text{sol}}$ of $V(F_\infty)$ consisting of elements soluble at all the places in $M_\infty$ and then count the $V_\beta$-points in $\FF$. We construct $\FF$ from a nicely shaped fundamental domain for the action of $G(F_\infty)$ on $V(F_\infty)^{\text{sol}}$ and from a fundamental domain for the action of $G_\beta$ on $G(F_\infty)$ using the reduction theory of reductive groups (\cite{Sp}). We construct the set $G(F_\infty)\backslash V(F_\infty)$ so that the subsets of elements having bounded height are homogeneously expanding.

We break up $G_\beta\backslash G(F_\infty)$ into two pieces: a compact part which we call the {\it main body}; and a non-compact part which we call the {\it cuspidal region} or {\it cusp}. For the main body, we count the lattice points using a suitable generalization of Davenport's lemma (Proposition \ref{prop:Davenport}). For the cuspidal region, the error term in approximating the number of lattice points by volumes is larger than the main term. We slice the cuspidal region into finer slices and estimate the total contribution to the total number of lattice points coming from each slice. The slicing and the estimates for the cuspidal region have been carried out over $\Q$ and were a crucial step in obtaining the arithmetic statistical results over $\Q$. They remain among the most important steps over arbitrary global fields. We reduce the results needed to some combinatorial conditions on the characters of the maximal split torus of $G$. In particular, when $G$ is split over $\Z$ and the representation $V$ of $G$ is base-changed from $\Z$ as is the case in all our applications in this paper, we prove that the validity of these combinatorial conditions is purely algebraic, and is independent of the global field $F$. 

Once the integral orbits are counted, we sieve to the locally soluble rational orbits by counting integral orbits where each orbit is weighted by a carefully constructed weight function. This weight function encodes the following information as a product of local functions: the breaking of one rational orbit into multiple integral orbits; the difference in the rational stabilizers and the integral stabilizers; and the local solubility conditions.  
Standard sieving techniques give an upper bound for the number of orbits. To show that it is also a lower bound, uniformity estimates for the error terms are required. We prove the desired uniformity estimates for the representations related to Selmer groups of elliptic curves, but not for the ones related to Selmer groups of Jacobians of hyperelliptic curves (which are currently unproved for the case $F=\Q$ as well). As part of our technical results, we prove that the techniques for obtaining these uniformity estimates over $\Q$ carry over to the general case of global fields. 


Once the counting and sieving results are proved, it only remains to compute the product of the various local volumes that arise. At places in $M_\infty$, these are the volumes of sets that arise when applying geometry of numbers methods. At finite places, these are the volumes of the local weight functions that encode our desired conditions. When multiplied together, these local volumes admit a lot of cancellation, due to the product formulas for global fields and abelian varieties. All that remains is the Tamagawa number of $G$ over $F$. Adding in the contribution from (the previously excluded) non-generic orbits then gives the total number of locally soluble orbits, completing the proofs.


\medskip

Field extensions of low-degrees (degrees $2$, $3$, $4$, and $5$) are parametrized by orbits in prehomogeneous vector spaces. The count of these low degree field extensions of $\Q$, ordered by discriminant, were carried out in the works of Davenport--Heilbronn \cite{DH} and the first-named author \cite{dodqf,dodpf}. In the previous article \cite{BSWfield} of this series, we generalized these works from the setting $F=\Q$ to general global fields $F$. This article, however, requires a number of additional ideas due to some important differences between the prehomogeneous and the general coregular cases.

First, in the prehomogeneous cases, each $G(F)$-orbit contains an element of $V_\beta$ for a unique $\beta$ (corresponding to the Steinitz class of the field extension). Thus we needed to count orbits in $V_\beta$ and then sieve to maximal ones. In the general coregular cases studied in this paper, however, this uniqueness does not hold. Moreover, there is no clear subset that needs to be counted. Instead, we construct weight functions at each local place, and then carry out weighted counts.

Furthermore, in the prehomogeneous cases \cite{BSWfield}, the fundamental set $G(F_\infty)\backslash V(F_\infty)$ is relatively simple (in fact just a finite set) and we only needed $G_\beta\backslash G(F_\infty)$ to be nicely shaped. In the coregular case, the needed fundamental sets have to be carefully constructed, to ensure that as our height increases, we are counting lattice points in homogeneously expanding bounded domains.
The ``cutting off the cusp'' part of the geometry-of-numbers input was done in a case-by-case manner in \cite{BSWfield}. In this paper, we give an axiomatic description of the necessary combinatorial inputs for this step, in order to carry out the counts over any global field.

The sieving and uniformity steps in the prehomogeneous cases were again relatively simple --- there were discriminant lowering operations (which we discuss here as ``reduction sieves''), which when combined with the counting step, yielded the required results relatively easily. In our case, the uniformity estimates are significantly more difficult, and require a combination of the reduction sieves, generalizations of the Ekedhal sieve, and ``embedding sieves''.

Many further results in the literature, including those of \cite{BH1, LR, Laga, RT, Siad1, Siad2, Swaminathan1, Th, Th1}, may also be extended to general global fields using the techniques of this paper and its prequel. More specifically, techniques in \cite{BSWfield} can be used to count orbits on prehomogeneous representations of reductive groups over global fields, while the techniques developed in this article can be used to count orbits on coregular representations of semisimple groups over global fields.

\medskip

This paper is organized as follows. In Section \ref{sec:notation}, we fix our notations, and prove some preliminary results on counting integral points in affine and weighted projective spaces. In Section \ref{sec:general}, we work in the most general setting of a coregular representation of a reductive group. We present a list of axioms, and prove a general formula for the number of orbits (with bounded height) which holds once this list of axioms is known to be satisfied. In Section \ref{sec:count}, we then describe the various methods we develop for proving these axioms. Finally, in Sections \ref{sec:elliptic} and \ref{sec:hyper}, we apply our general results to the cases of elliptic curves and hyperelliptic curves, respectively, and prove our main theorems.

\section{Preliminaries}\label{sec:notation}
In this section, we introduce notation and the height functions that will be used throughout this paper. We also prove some results on counting lattice points in affine and weighted projective spaces.

\subsection{Notation}\label{sec:not}
Let $F$ denote a fixed global field. That is, $F$ is either a number field or the field of rational functions of a smooth projective and geometrically connected algebraic curve $\sc$ over $\F_q$. When $F$ is a number field, let $\cO$ denote the ring of integers of $F$ and let $M_\infty$ denote the set of archimedean places of $F$. When $F$ is the field of rational functions of a smooth projective and geometrically connected algebraic curve $\sc$ over $\F_q$, fix some nonempty finite set $D$ of closed points on $\CC$, let $\cO$ denote the ring of regular functions on $\sc-D$ and let $M_\infty$ denote the finite set of places corresponding to $D$. 

Let $\A$ denote the ring of adeles. Let $\A_f$ denote the ring of finite adeles, that is the restricted direct product of $F_\fP$ over primes $\fP\notin M_\infty$ and let $F_\infty$ denote the product of the completions $F_\fP$ for $\fP\in M_\infty$.

For every non-archimedean place $w$ of $F$, denote by $F_w$, $\cO_w$, and $k(w)$ the completion of $F$ at $w$, the ring of integers at $w$ and the residue field at $w$, respectively. Let $Nw=|k(w)|$ denote the norm of $w$. For any $a\in F_w,$ define its $w$-adic norm by $|a|_w=(Nw)^{-w(a)}$ where $w(a)$ denotes the $w$-adic valuation of $a$. Define $\deg w = [k(w):\F_q]$.

When $F$ is a number field and $w$ is archimedean, we use the usual absolute value for the norm $|\cdot|_\infty$ on $\R$ and define $|a|_w=|N_{F_w/\R}(a)|_\infty$ for any $a\in F_w$. We have the product formula $\prod_w |a|_w=1$ for any $a\in F^\times$. We define $\deg w = 2$ if $F_w = \C$ and $\deg w = 1$ if $F_w = \R$.

In both cases, we define $$d_F = \sum_{w\in M_\infty}\deg w.$$
Then $d_F$ equals $\dim_\R F_\infty = [F:\Q]$ when $F$ is a number field; and equals the number of points in $D$ over the algebraic closure when $F$ is a function field.

For any adele $a=(a_\fP)_\fP\in\A$, we write $|a|$ for the adele norm $\prod_{\fP}|a_\fP|_\fP.$ We also view $\A_f$ and $F_\infty$ as subrings of $\A$ and use $|\cdot|$ to denote the restriction of the adele norm to these subrings.

\subsection{Heights and fundamental domains for weighted projective spaces}\label{sec:heightweighted}

In this subsection, we discuss some generalities regarding weighted projective spaces. This will be used in Sections \ref{sec:elliptic} and \ref{sec:hyper}, where we view the coefficients of defining equations of elliptic curves and hyperelliptic curves, respectively, as points in weighted projective space.

Let $w_1,\ldots,w_n$ be positive integers and let $\P(w_1,\ldots,w_n)$ denote the weighted projective space with weights $w_1,\ldots,w_n$. That is, consider the action of $\bG_m$ on $S=\A^n$ given by $\alpha.(A_1,\ldots,A_n) = (\alpha^{w_1}A_1,\ldots,\alpha^{w_n}A_n).$ Then $\P(w_1,\ldots,w_n)$ is the quotient $S(F)/\bG_m(F)$, and we denote it by $P(F)$. Note that the usual projective space $\P^m$ is simply $\P(1,1,\ldots,1)$ (with $m+1$ $1$'s) under this definition. Given any $(A_1,\ldots,A_n)\in S(F)$, let $I$ denote the $\cO$-ideal
$$I := \{\alpha\in F:\alpha.(A_1,\ldots,A_n)\in S(\cO)\};$$
we define the height $H$ on $S(F)$ to be
\begin{equation}\label{eq:heightweighted}
H(A_1,\ldots,A_n) := (NI)\prod_{\fP\in M_\infty} \max(|A_1|_\fP^{1/w_1},\ldots,|A_n|_\fP^{1/w_n}).
\end{equation}
The product formula implies that $H$ is invariant under the action of $\bG_m(F)$ and hence descends to a height function on $P(F)$, which we also denote by $H$. We extend the height function $H$ naturally to $S(F_\infty)$ in the natural way; that is, for any $(A_{1,\fP},\ldots,A_{n,\fP})_{\fP\in M_\infty}\in S(F_\infty)$,  define its height to be
$$H((A_{1,\fP},\ldots,A_{n,\fP})_{\fP\in M_\infty}) := \prod_{\fP\in M_\infty}H_\fP(A_{1,\fP},\ldots,A_{n,\fP})$$
where $$H_\fP(A_{1,\fP},\ldots,A_{n,\fP}):=\max(|A_1|^{1/w_1}_\fP,\ldots,|A_n|^{1/w_n}_\fP).$$

The goal of this section is to carry out the following construction.

\begin{proposition}\label{prop:Sigma}
    There exists a fundamental domain $\Sigma$ for the action of $\bG_m(F)$ on $S(F)$ of the form 
    $$\Sigma = S(F)\cap\left(\prod_{\fP\notin M_\infty}\Sigma_\fP \times \Sigma_\infty\right),$$
    where the intersection is taken in $S(\A)$ and where
    \begin{enumerate}
        \item $\Sigma_\fP\subset S(\cO_\fP)$ is open for $\fP\notin M_\infty$, all but finitely many of which equal $S(\cO_\fP)$;
        \item $\Sigma_\infty\subset S(F_\infty)$ is measurable such that for any $\fP,\fP'\in M_\infty$ and any $(A_{1,\fP},\ldots,A_{n,\fP})_{\fP\in M_\infty}\in\Sigma_0$,
        $$1\ll\frac{H_\fP(A_{1,\fP},\ldots,A_{n,\fP})^{1/\deg \fP}}{H_\fP(A_{1,\fP'},\ldots,A_{n,\fP'})^{1/\deg \fP'}}\ll 1,$$
        where the implied constants depend only on $F$ and $M_\infty$.
    \end{enumerate}
    As a consequence, for any $\lambda >0$, the set $\{f\in \Sigma_\infty\colon H(f)\leq \lambda\}$ is compact.
\end{proposition}



Let $\fP$ be a finite prime of $F$, and let $\nu_\fP$ denote the corresponding valuation. We define $$S'(\cO_\fP) = \{(A_1,\ldots,A_n)\in S(\cO_\fP):\nu_\fP(A_i) < w_i \mbox{ for some }i=1,\ldots,n\}.$$
Then for any $v \in S(F_\fP)$, there exists $g_\fP\in \bG_m(F_\fP)$ such that $g_\fP.v \in S'(\cO_\fP).$ Moreover, if $g\in \bG_m(F_\fP)$ sends an element of $S'(\cO_\fP)$ to $S'(\cO_\fP)$, then $g\in \bG_m(\cO_\fP).$ Let $$\cl(\bG_m)=\big(\prod_{\fP\notin M_\infty} \bG_m(\cO_\fP)\big)\backslash \bG_m(\A_f)/\bG_m(F)$$
denote the class group of $\bG_m$ over $F$. Then $\cl(\bG_m)$ is isomorphic to the class group of $F$, and hence is finite. For any $\gamma\in\cl(\bG_m)$, define
\begin{eqnarray}
\label{eq:S_gamma} S_\gamma &=& S(F)\cap \gamma^{-1}\big(\prod_{\fP\notin M_\infty} S'(\cO_\fP)\big),
\end{eqnarray}
where the intersection is in $S(\A_f)$. We fix representatives $\gamma=(\gamma_v)_{v\in M}\in \bG_m(\A_f)$ for each element in $\cl(\bG_m)$. Multiplying by elements of $\bG_m(F)$, if necessary, we may assume that for any $\gamma$, we have $\nu_\fP(\gamma_\fP)\leq 0$ for every $\fP\notin M_\infty$. Then $S_\gamma$ lies in the image of $S(\cO)$ in $S(\A_f)$, and henceforth, we identify $S_\gamma$ with its preimage in $S(\cO)$. Note that since $S(F)$ and $S'(\cO_\fP)$ are invariant under the action of $\bG_m(\cO)=\cO^\times$ for every $\fP$, the set $S_\gamma$ is also invariant under the action of $\cO^\times$. We have the following lemma:
\begin{lemma}
There is a bijection
$$P(F) = \bG_m(F)\backslash S(F) \longleftrightarrow \bigsqcup_\gamma \cO^\times\backslash S_\gamma,$$
where the union goes over our fixed set of representatives $\gamma$ for $\cl(\bG_m)$.
\end{lemma}
\begin{proof}
First we define the map. Let $v$ be an element of $S(F)$. For all $\fP\not\in M_\infty$, there exist elements $\gamma_\fP\in\bG_m(F_\fP)$, unique up to $\bG_m(\cO_\fP)$, such that $\gamma_\fP\cdot v\in S'(\cO_\fP)$.
Write the element $(\gamma_\fP)_{\fP\not\in M_\infty}\in \bG_m(\A_f)$ as $\gamma_i\gamma\gamma_F$, where $\gamma_i\in \prod_{\fP\not\in M_\infty}\bG_m(\cO_\fP)$, $\gamma$ is one of our fixed representatives for $\cl(\bG_m)$, and $\gamma_F\in \bG_m(F)$. It immediately follows that $\gamma_F v$ belongs to $S_\gamma$. Define our map by sending $v$ to the $\cO^\times$-orbit of $\gamma_F v$. This map is clearly well defined, and surjectivity follows by noting that elements in $S_\gamma$ map to themselves. It remains to check injectivity. To this end, note that if $v\in S_\gamma$ and $g v\in S_{\gamma'}$ for some $g\in \bG_m(F)$, then the elements $\gamma_\fP v$ and $\gamma_\fP' g v$ belong to $S'(\cO_\fP)$ for all $\fP\not\in M_\infty$. Hence $\gamma_\fP' g\gamma_\fP^{-1}\in \bG_m(\cO_\fP)$, implying that $\gamma$ and $\gamma'$ represent the same class in $\cl(\bG_m)$.
\end{proof}

It suffices to consider each $\gamma\in\cl(\bG_m)$ separately. We now fix some $\gamma$. If $(A_1,\ldots,A_n)\in S_\gamma$, then $$H(A_1,\ldots,A_n) = \prod_{\fP\in M_\infty}\max(|A_1|_\fP^{1/w_1},\ldots,|A_n|_\fP^{1/w_n}).$$ We construct a fundamental domain $\Sigma$ for the action of $\cO^\times$ on $S_\gamma$ by taking $\Sigma_\fP = \gamma_\fP^{-1}S'(\cO_\fP)$ for $\fP\notin M_\infty$ and by constructing $\Sigma_\infty$ as a fundamental domain for the action of $\cO^\times$ on $S(F_\infty)$ such that for any $(A_{1,\fP},\ldots,A_{n,\fP})_{\fP\in M_\infty}\in \Sigma_\infty,$ the local heights $H_\fP(A_{1,\fP},\ldots,A_{n,\fP})^{1/\deg\fP}$ differ from each other by at most an absolute constant. 
We treat the number field and the function field cases separately due to the valuation set $|F_\infty|$ being discrete for the latter. 


Suppose first that $F$ is a number field. Let $F_\infty^1$ denote the subset of $F_\infty$ consisting of elements $(\alpha_\fP)$ such that $\prod_{\fP\in M_\infty} |\alpha_\fP|_\fP=1$. Let $\Lambda$ denote a compact subset of $F_\infty^1$ such that $F_\infty^1=\Lambda.\cO^\times.$ Fix any $(A_{1,\fP},\ldots,A_{n,\fP})_{\fP\in M_\infty}\in S(F_\infty).$ Let
$$H'((A_{1,\fP},\ldots,A_{n,\fP})_{\fP\in M_\infty}) = \prod_{\fP\in M_\infty}H_\fP(A_{1,\fP},\ldots,A_{n,\fP})^{1/\deg\fP}.$$
For any $\fP\in M_\infty$, let $\alpha_\fP$ be an element of $F_\fP$ such that
\begin{equation*}
 |\alpha_\fP|_\fP = \frac{H_\fP(A_{1,\fP},\ldots,A_{n,\fP})}{H'((A_{1,\fP},\ldots,A_{n,\fP})_{\fP\in M_\infty})^{\deg\fP/|M_\infty|}}.
\end{equation*}
Then $(\alpha_\fP)_{\fP\in M_\infty}\in F_\infty^1$ and for any $\fP\in M_\infty$, we have $$\left(H_\fP(\alpha_\fP^{-1}.(A_{1,\fP},\ldots,A_{n,\fP}))\right)^{1/\deg\fP} = H'((A_{1,\fP},\ldots,A_{n,\fP})_{\fP\in M_\infty})^{1/|M_\infty|}.$$ 
Let $S(1)$ denote the set of elements $(A_{1,\fP},\ldots,A_{n,\fP})_{\fP\in M_\infty}\in S(F_\infty)$ such that $$H_\fP(A_{1,\fP},\ldots,A_{n,\fP})^{1/\deg\fP}=H_\fP(A_{1,\fP'},\ldots,A_{n,\fP'})^{1/\deg\fP'}\qquad\mbox{for all }\fP,\fP'\in M_\infty.$$
Then  $$S(F_\infty)= F_\infty^1.S(1).$$
Now if $\alpha\in \cO^\times$ sends an element of $S(1)$ back to $S(1)$, then $$|\alpha|_\fP^{1/\deg\fP} = |\alpha|_{\fP'}^{1/\deg\fP'}\mbox{ for all }\fP,\fP'\in M_\infty\mbox{ and }\prod_{\fP\in M_\infty}|\alpha|_\fP = 1,$$ from which we see that $|\alpha|_\fP = 1$ for all $\fP\in M_\infty$ and hence $\alpha$ is a root of unity. Let $\mu_\infty$ be the finite group of roots of unity contained in $F$ and let $\mu_\infty\backslash S(1)$ be a measurable set whose boundary has measure $0$ that is a fundamental domain for the action of the finite group $\mu_\infty$ on $S(1)$. We take $\Sigma_\infty = \Lambda.(\mu_\infty\backslash S(1))$.

Suppose next that $F$ is a function field. For any $\fP\in M_\infty$, let $c_\fP $ be the positive absolute constant so that for any $r\in \R^+$, the intersection $[r/c_\fP ,rc_\fP )\cap |F_v|_\fP$ is a singleton. For example, if $F=\F_q(T)$, then $c_{1/T} = \sqrt{q}$. We write $c_0$ for the product $\prod_{\fP\in M_\infty} c_\fP $. Let $F_\infty^{\sim1}$ denote the set of $(a_\fP)_{\fP\in M_\infty}\in F_\infty$ such that $c_0^{-1} \leq \prod |a_\fP|_\fP < c_0$. Let $\Lambda'$ denote a finite union of translates of $\Lambda$ such that $F_\infty^{\sim1}=\Lambda'.\cO^\times.$ Fix any $(A_{1,\fP},\ldots,A_{n,\fP})_{\fP\in M_\infty}\in S(F_\infty).$ For any $\fP\in M_\infty$, let $\alpha_\fP\in F_\fP$ be an element such that $$\frac{1}{c_\fP }\leq |\alpha_\fP|_\fP  \frac{(H'((A_{1,\fP},\ldots,A_{n,\fP})_{\fP\in M_\infty}))^{\deg\fP/|M_\infty|}}{H_\fP(A_{1,\fP},\ldots,A_{n,\fP})}< c_\fP .$$
In contrast to the number field case, we have $(\alpha_\fP)_{\fP\in M_\infty}\in (F_\infty)^{\sim 1}$ and for any $\fP\in M_\infty$, we have $$c_\fP^{-1/\deg\fP}\leq \frac{H_\fP(\alpha_\fP^{-1}.(A_{1,\fP},\ldots,A_{n,\fP}))^{1/\deg\fP}}{H'((A_{1,\fP},\ldots,A_{n,\fP})_{\fP\in M_\infty})^{1/|M_\infty|}} \leq c_\fP^{1/\deg\fP} .$$ 
Let $S(\sim\!\! 1)$ denote the set of elements $(A_{1,\fP},\ldots,A_{n,\fP})_{\fP\in M_\infty}\in S(F_\infty)$ such that $$c_\fP^{-1/\deg\fP} c_{\fP'}^{-1/\deg\fP'}\leq\frac{H_\fP(A_{1,\fP},\ldots,A_{n,\fP})^{1/\deg\fP}}{H_{\fP'}(A_{1,\fP'},\ldots,A_{n,\fP'})^{1/\deg\fP'}}\leq c_\fP^{1/\deg\fP} c_{\fP'}^{1/\deg\fP'}\qquad\mbox{for all }\fP,\fP'\in M_\infty.$$
Then  $$S(F_\infty) = F_\infty^{\sim 1}.S(\sim\!\! 1).$$
The subset $\mu_\infty'$ of $\cO^\times$ consisting of elements $\alpha$ such that $\alpha.S(\sim\!\! 1)\cap S(\sim\!\! 1)\neq\emptyset$ is a finite set as it is a finite union of translates of $\mu_\infty$. Two elements of $S(\sim\!\! 1)$ are called \textit{equivalent} if there exists an element of $\mu_\infty'$ sending one to the other. Let $\mu_\infty'\backslash S(\sim\!\! 1)$ denote a measurable set consisting of one representative from each equivalence class and take $\Sigma_\infty = \Lambda'.(\mu_\infty'\backslash S(\sim\!\! 1))$.

The proof of Proposition \ref{prop:Sigma} is now complete.

\begin{remark}\label{rem:S1}
    In the definition of $S(1)$ above, we may choose to have
    $$H_\fP(A_{1,\fP},\ldots,A_{n,\fP})^{n_\fP}$$
    be equal across $\fP\in M_\infty$, where the $n_\fP$'s are some fixed positive real numbers. Our choice of $n_\fP = 1/\deg\fP$ gives an optimal power-saving error term, so that when identifying $F_\infty^n = \R^{n[F:\Q]}$, the coordinates of elements in $\Sigma_\infty$ are bounded multiplicatively from each other by $O(1)$.
\end{remark}

\subsection{Counting lattice points over $F_\infty$}

In this section, we collect some results on geometry-of-numbers over global fields. We have the following version of Davenport's lemma over global fields (\cite[Proposition 4.4 and Proposition A.4]{BSWfield}).

\begin{proposition}\label{prop:Davenport}
 If $F$ is a number field, let $E = \R$ and let $B$ be an open bounded and
semialgebraic subset of $E^n$. If $F$ is a function field, let $E = F_\infty$ and let $B$ be an open compact subset of $E^n$. Let $K$ be an open compact subset of $\GL_n(E)$. Let $c$ be a positive real constant. Then for any $g\in K$ and for any $t=
 \diag(t_1,\ldots,t_n)\in\GL_n(E)$ such that for any $i=1,\ldots,n$, we have $|t_i|_w^{1/\deg w}/|t_i|_{w'}^{1/\deg w'}<c$ for any $w,w'\in M_\infty$,
 \begin{equation}\label{eq:skew}
 \#\{tgB \cap \cO^n\} = \Vol_{\cO^n}(tgB) + O(\Vol(\proj(tgB))),
 \end{equation}
 where $\Vol_{\cO^n}$ is normalized such that
$E^n/\cO^n$ has volume $1$, and $\Vol(\proj(tgB))$ denotes the greatest $d$-dimensional volume of any projection of $tgB$ onto a coordinate subspace obtained by equating $n-d$ coordinates to zero, where $d$ takes all values from $1$ to $n-1$. The implied constant in the second summand depends only on $F$, $B$, $K$ and $c$.
\end{proposition}

The next result states that counting with a mod $\fP$ condition follows as one ``expects''.
\begin{proposition}\label{prop:cong}
    Let $B$ be a compact region in $F_\infty^n$ and let $r\in \R$ if $F$ is a number field and $r\in F$ if $F$ is a function field. Let $Y$ be a closed subscheme of $\A_{\cO}^n$. Let $I$ be an ideal of $\cO$. Then
    $$\#\{a\in rB\cap \cO^n\colon a \,(\mathrm{mod}\,I) \in Y(\cO/I)\} = O\left(\Big(\frac{|r|^n}{(NI)^n}+1\Big)\#Y(\cO/I)\right).$$
    Here $|r| = \prod_{w\in M_\infty}|r|_w$ so that $\Vol(rB) = |r|^n\cdot\Vol(B).$ The implied constant depends only on $B$.
\end{proposition}

\begin{proof}
    We cover $rB$ by cartesian products $D=\prod_{w\in M_\infty} \prod_{i=1}^n D_{w,i}$ of balls such that for any $a,b\in D_{w,i}$, we have $|a-b|_w < (NI)^{1/\#M_\infty}.$ If $(a_1,\ldots,a_n)$ and $(b_1,\ldots,b_n)\in \cO^n\cap D$ are such that $a_i$ and $b_i$ are congruent modulo $I$ for any $i=1,\ldots,n$, then
\begin{eqnarray*}
 |a_i-b_i|_\fP &\leq& (N\fP)^{-\nu_\fP(I)},\mbox{ for any }\fP\notin M_\infty\\
 |a_i-b_i|_w &<& (NI)^{1/\#M_\infty}, \mbox{ for any }w\in M_\infty.
\end{eqnarray*}
where $\nu_\fP(I)$ denote the $\fP$-adic valuation of the ideal $I$.
Hence $\prod_w |a_i-b_i|_w < 1$ which implies that $a_i=b_i$. That is, there is at most one element in $\cO^n\cap D$ for each congruence class modulo $I$. The number of points in $\cO^n\cap D$ whose reduction modulo $I$ is in $Y(\cO/I)$ is then bounded by $\#Y(\cO/I).$

We claim that $rB$ can be covered by $O((|r|/NI)^n+1)$ such $D$'s. To cover $rB$, we take a maximal filling of $rB$ by disjoint balls of the form $D'=\prod_{w\in M_\infty} \prod_{i=1}^n D'_{w,i}$ such that for any $a,b\in D'_{w,i}$, we have $|a-b|_w < \frac{1}{2}(NI)^{1/\#M_\infty}$. Expanding each $D'$ by a factor of $2$ in all coordinates gives a covering of $rB$. The number of balls needed is $O(\Vol(rB)/\Vol(D')+1)=O((|r|/NI)^n+1).$
\end{proof}

\begin{remark}
    Note that in the case $r\in \R$, the adelic norm $|r|$ equals the usual absolute value of $r$ to the power $[F:\Q]$. So the sizes of each coordinate of $a\in rB$ can be as small as $(N\fP)^{1/[F:\Q]}$ for the main term in Proposition \ref{prop:cong} to still dominate. This can also be explained using the fact that the successive minima of an ideal $I$ viewed as a lattice in $\R^{[F:\Q]}$ are bounded by each other multiplicatively by $O_F(1)$. To see this, let $\alpha\in I$ be the first element in a Minkowski basis of $I$. Let $\{\beta_1,\ldots,\beta_{[F:\Q]}\}$ be a Minkowski basis for $\cO$. Then for each $i = 1,\ldots,[F:\Q]$, we have $\|\alpha\beta_i\| =O_F( \|\alpha\|).$ From this we conclude that all the successive minima of $I$ are $O_F(\|\alpha\|)$.
\end{remark}

\section{A general framework for counting and sieving problems}\label{sec:general}


In this section, we present the main framework via which our counting and sieving results will be proved. Specifically, we present a list of axioms that we would like our representation $(G,V)$ to satisfy over our global field $F$. Finally, we present the main counting result implied by all these axioms. Unless specified otherwise, everything following an axiom assumes the validity of that axiom.

\medskip

\noindent Let $G$ be an algebraic group with a representation $V$ defined over $\cO$.

\vspace{10pt}
\noindent \textbf{AXIOM: \nicerep}. The representation $(G,V)$ satisfies the following conditions.
\begin{enumerate}
 \item $G$ is semisimple.
 \item The ring of polynomial invariants for the action of $G(\cO)$ on $V(\cO)$ is freely-generated by homogeneous polynomials $I_1,\ldots,I_k$ in the coordinates of $V$.
 \item The sum of the homogeneous degrees of $I_1,\ldots,I_k$ is equal to $\dim V$.
 \item The generic stabilizer is finite.
\end{enumerate}
\vspace{10pt}

Let $S$ denote the canonical quotient $V/\!\!/G$ defined as Spec of the ring of polynomial invariants over $\cO$. Then AXIOM: \nicerep~implies that $S$ is an affine space of dimension $$k=\dim\; S=\dim\; V - \dim\; G.$$ Let $\inv:V\rightarrow S$ denote the quotient map. Let $H:S(F)\rightarrow \R_{\geq0}$ denote a height function on $S$. Composing $H$ with $\inv$ gives a height function on $V(F)$ which we denote also by $H$. Let  $V(F)^{\gen}$ be a fixed $G(F)$-invariant subset of $V(F)$. Elements of $V(F)^\gen$ are called \emph{generic} and elements not in $V(F)^\gen$ are called \emph{non-generic}. Our goal is to obtain asymptotics for the number of generic $G(F)$-orbits on $V(F)$ satisfying certain congruence properties with integral invariants and bounded height.

For any positive integer $M$, we say that a function $\phi: F^M \rightarrow [0,1]$ is \emph{defined by congruence conditions} if there exist local functions $\phi_\fP:F^M_\fP\rightarrow [0,1]$ for every place $\fP$ of $F$ (including archimedean places) such that:
\begin{enumerate}
 \item For all $v\in F^M,$ the product $\prod_\fP \phi_\fP(v)$ converges to $\phi(v)$.
 \item For each $\fP$, the function $\phi_\fP$ is locally constant outside of some $\fP$-adically closed subset of $F_\fP^M$ of measure $0$.
\end{enumerate}
A subset of $F^M$ is said to be \emph{defined by congruence conditions} if its characteristic function is defined by congruence conditions.

\vspace{10pt}
\noindent \textbf{AXIOM: \localCondition}. Let $m_0=\prod_\fP m_{0,\fP}$ be a weight function on $V(F)$ defined by congruence conditions such that:
\begin{enumerate}
 \item $m_{0,\fP}$ is $G(F_\fP)$-invariant;
 \item For any $\fP\in M_\infty$, any $v\in V(F_\fP)$ and any $\lambda\in F_\fP^\times$, we have $m_{0,\fP}(v)\neq 0$ if and only if $m_{0,\fP}(\lambda v)\neq 0$;
 \item For any $v\in V(F)$ with $m_0(v)\neq0$, the stabilizer $\Stab_{G}(v)=\{g\in G:gv = v\}$ is a finite group scheme with absolutely bounded order and for every finite prime $\fP$, there exists $g_\fP\in G(F_\fP)$ such that $g_\fP v\in V(\cO_\fP)$;
 \item The product $\prod_{\fP\notin M_\infty}\lambda_\fP' = 0$, where $\lambda_\fP'$ is the $\fP$-adic density of $v\in V(F_\fP)$ with $\#\Stab_{G(F_\fP)}(v)\neq 1$ and $m_{0,\fP}(v)\neq 0$
\end{enumerate}
We write $V_P(F)$ for the subset of $V(F)$ consisting of elements $v$ with $m_0(v)\neq0$. Define $\Sigma := \inv(V_P(F)).$
\vspace{10pt}

For any subgroup $G_0$ of $G(F)$, any $G_0$-invariant subset $V_0$ of $V(F)$, and any $G_0$-invariant function $m:V_0\rightarrow[0,1]$, let $N_m(V_0,G_0,X)$ denote the number of generic $G_0$-orbits on $V_0$ of height bounded by $X$ when $F$ is a number field or equal to $X$ when $F$ is a function field, where each orbit $G_0v$ is weighted by $m(v)/(\#\Stab_{G_0}(v))$. If $m$ is identically $1$ on $V_0$, we write $N(V_0,G_0,X)$ instead. 
We are interested in $N_{m_0}(V(F),G(F),X)$. In most applications, $m_0$ will be the characteristic function of a subset of $V(F)$ defined by congruence conditions. Condition 4 of AXIOM: \localCondition~implies that the extra weighting by $1/\#\Stab_{G_0}(v)$ does not affect the asymptotics when $G_0$ is commensurable with $G(\cO)$.

We first reduce the evaluation of $N_m(V_0,G_0,X)$ to counting integral orbits. Let $$\cl(G)=\Big(\prod_{\fP\notin M_\infty} G(\cO_\fP)\Big)\backslash G(\A_f)/G(F)$$
denote the class group of $G$ over $F$. Finiteness of $\cl(G)$ is proved by Borel (\cite{Borel}) in the case of number fields, by Borel and Prasad (\cite{BP}) in the function field case when $G$ is adjoint semisimple, and by Conrad (\cite{Conrad}) in the general function field case. For any $\beta\in\cl(G)$, fix a representative in $G(\A_f)$ which we denote also by $\beta$, and define
\begin{eqnarray}
\label{eq:V_beta} V_\beta &:=& V(F)\cap \beta^{-1}\big(\prod_{\fP\notin M_\infty} V(\cO_\fP)\big),\\
\label{eq:G_beta} G_\beta &:=& G(F)\cap \beta^{-1}\big(\prod_{\fP\notin M_\infty} G(\cO_\fP)\big)\beta,
\end{eqnarray}
where the intersections above are taken in $V(\A_f)$ and $G(\A_f)$, respectively. Then $G_\beta$ is commensurable with $G(\cO)$ and acts naturally on $V_\beta$. 

Given any $v\in V(F)$ with $m_0(v)\neq 0$ and any $\fP\notin M_\infty$, there exists $g_\fP\in G(F_\fP)$ such that $g_\fP v\in V(\cO_\fP).$ The adele $(g_\fP)_\fP\in G(\A_f)$ can be written as $(g'_\fP)_\fP\beta g$ where $g'_\fP\in G(\cO_\fP)$ for any $\fP\notin M_\infty$, where $\beta$ is the fixed representative of some element of $\cl(G)$ in $G(\A_f)$ and $g\in G(F)$. Then $gv\in V_\beta$. We have the formula
\begin{equation}\label{eq:breakupclG}
 N_{m_0}(V(F),G(F),X) = \sum_\beta N_m(V_\beta, G_\beta, X),
\end{equation}
where the weight function $m$ is defined by
\begin{equation}\label{eq:defnofm}
 m(v) = \frac{m_0(v)}{\#\Stab_{G(F)}(v)}\left(\sum_\beta \sum_{v_{\beta,i}} \frac{1}{\#\Stab_{G_\beta}(v_{\beta,i})}\right)^{-1};
\end{equation}
here $\{v_{\beta,i}\}$ denotes a complete set of representatives for the action of $G_\beta$ on $G(F)v \cap V_\beta$.

\begin{lemma}\label{lem:mcongruence}
 The weight function $m$ in \eqref{eq:defnofm} is defined by congruence conditions.
\end{lemma}

\begin{proof}
 For any $\fP\notin M_\infty$, define $m_\fP$ to be
\begin{equation}\label{eq:defnofmp}
 m_\fP(v) = \frac{m_{0,\fP}(v)}{\#\Stab_{G(F_\fP)}(v)}\left(\sum_{v_i} \frac{1}{\#\Stab_{G(\cO_\fP)}(v_{i})}\right)^{-1},
\end{equation}
where $\{v_i\}$ is a complete set of representatives for the action of $G(\cO_\fP)$ on $G(F_\fP)v\cap V(\cO_\fP).$ For $\fP\in M_\infty$, we set $m_\fP$ to be $m_{0,\fP}(v).$ Then for any $v\in V_P(F)$,
$$m(v) = \prod_\fP m_\fP(v)$$
by \cite[Theorem 4.3.1]{ST}.
\end{proof}

We now turn to estimating $N_m(V_\beta,G_\beta,X)$ for some fixed $\beta\in \cl(G)$, or, more generally, to the estimation of $N_m(V_0,G_0,X)$ where $G_0$ is a subgroup of $G(F)$ commensurable with $G(\cO)$, and $V_0$ is a $G_0$-invariant lattice in $V(F)$ commensurable with $V(\cO)$. Let $m_\infty = \prod_{\fP\in M_\infty}m_\fP$ be the product of all the weights at infinity. We compute $N_{m_\infty}(V_0,G_0,X)$ first and then introduce the weights at finite primes from the viewpoint of congruence conditions. We write $V_P(F_\infty)$ for the subset of $V(F_\infty)$ consisting of elements $v$ with $m_{0,\fP}(v)\neq0$ for all $\fP\in M_\infty$, where $F_\infty = \prod_{\fP\in M_\infty}F_\fP$.  We 
construct first a fundamental domain for the action of $G_0$ on $V_P(F_\infty)$.

We say a continuous function $H$ on $V(F_\infty)$ is \emph{homogeneous of degree $d$} if for every $r\in F$ and every $v\in V(F_\infty)$,
$$H(rv) = H(v)\prod_{\fP\in M_\infty}|r|_\fP^d.$$

\vspace{10pt}
\noindent \textbf{AXIOM: Counting at Infinity I}: The height function $H$ when restricted to $\Sigma$ extends to a continuous function $H:S(F_\infty)\rightarrow \R_{\geq 0}$ such that when composed with $\inv:V(F_\infty)\rightarrow S(F_\infty)$, the resulting function on $V(F_\infty)$, still denoted $H$, is homogeneous of degree $d$. Moreover, for any $\lambda\in\R_{>0}$, there exists a pre-compact fundamental domain $R_\lambda$ for the action of $G(F_\infty)$ on $V_P(F_\infty)_{\lambda} = \{v\in V(F_\infty)\mid m_\infty(v)\neq0, H(v) = \lambda\}.$ \hfill$\Box$
\vspace{10pt}

The existence of $R_\lambda$ for all $\lambda$ is equivalent to the existence of $R_1$ when $F$ is a number field since we may take $R_\lambda = \lambda^{1/(d[F:\Q])}R_1$.

When $F$ is a number field, we write $R(X)=[1,X^{1/(d[F:\Q])}).R_1\cap V_P(F_\infty)$. Then $R(X)$ is a fundamental domain for the action of $G(F_\infty)$ on the subset $V_P(F_\infty)_{<X}$ of $V_P(F_\infty)$ consisting of elements of height bounded by $X$. When $F$ is a function field, we take $R(X)$ to be of the form $\Lambda_XR_\lambda\cap V_P(F_\infty)$ where $\Lambda_X\in F$ and $\lambda$ belongs to some fixed finite set of positive real numbers. The key is that as $X$ goes to infinity, $|\Lambda_X|$ also goes to infinity while $R_\lambda$ varies in a fixed finite set, so that $R(X)$ is a fundamental domain for the action of $G(F_\infty)$ on the subset $V_P(F_\infty)_{X}$. Note that in this case, we only consider real numbers $X$ such that $V(F_\infty)_{X}$ is nonempty. 

Let $\FF_0$ denote a fundamental domain for the action of $G_0$ on $G(F_\infty)$ by left multiplication. View $\FF_0.R(X)\subset V_P(F_\infty)_{<X}$ as a multiset where the multiplicity of $v\in \FF_0.R(X)$ is the cardinality of the set $\{g\in \FF_0:v\in gR(X)\}.$ Then $\FF_0.R(X)$ maps surjectively onto a fundamental domain for the action of $G_0$ on $V_P(F_\infty)_{<X}$, where the fiber above any orbit $G_0 v$ has size $\#\Stab_{G(F_\infty)}(v)/\#\Stab_{G_0}(v).$

Let $d\nu$ denote a left-invariant top differential on $V$ defined over $\cO$ and denote by $\nu_\infty$ and $\nu_\fP$ the induced measures on $V(F_\infty)$ and $V(F_\fP)$, for any place $\fP$, respectively. The volume of $V(\cO_\fP)$ computed with respect to $\nu_{\fP}$ is
$1$ for every finite prime $\fP$ and the covolume of $V(\cO)$ in $V(F_\infty)$ with respect to $\nu_\infty$
is $\sqrt{D_F}^{\,\dim V}$ where $D_F$ is the absolute discriminant of
$F$ (see \cite[\S2.1.1]{Weil}). 
Let $\nu_{\infty,0}$ denote a normalization of $\nu_\infty$ where the covolume of $V_0$ in $V(F_\infty)$ is $1$. For any measurable multiset $B\subset V(F_\infty)$ and any integrable function $\varphi$ on $B$, write
$$\nu_{\infty,0}^\varphi(B) = \int_{v\in B} \frac{\varphi(v)}{\#\Stab_{G(F_\infty)}(v)} d\nu_{\infty,0}(v).$$

\vspace{10pt}
\noindent \textbf{AXIOM: Counting at Infinity II}: View $\FF_0.R(X)$ as a multiset as above. Then for any subgroup $G_0$ of $G(F)$ commensurable with $G(\cO)$ and any $G_0$-invariant lattice $V_0$ in $V(F)$ commensurable with $V(\cO)$, we have:
\begin{equation}\label{eq:countingmain}
 N_{m_\infty}(V_0,G_0,X) \sim \nu_{\infty,0}^{m_\infty}(\FF_0.R(X)).
\end{equation}
\vspace{5pt}

Since every coordinate in $R(X)$ is homogeneously expanding with respect to $X$, AXIOM: Counting at Infinity II holds automatically if $\FF_0$ is pre-compact. This follows from Davenport's lemma (Proposition \ref{prop:Davenport}) on counting lattice points in a homogeneously expanding region. When $\FF_0$ is not compact and has cusps going to infinity, the averaging and cutting off the cusp techniques of \cite{BS2} are used to obtain this axiom. See Theorem \ref{prop:maintransference} where we reduce AXIOM: Counting at Infinity II to purely combinatorial conditions on the characters of the maximal split torus of $G$.

Since $\nu_\infty$ is normalized so that $V(\cO)$ has covolume $\sqrt{D_F}^{\,\dim V}$ and $V_0$ is commensurable with $V(\cO)$, Equation \eqref{eq:countingmain} can be rewritten as
\begin{equation}\label{eq:countingmainre}
 N_{m_\infty}(V_0,G_0,X) = \sqrt{D_F}^{\,-\dim V}\nu_{\infty}^{m_\infty}(\FF_0.R(X))\prod_{\fP\notin M_\infty}\int_{V_{0,\fP}}d\nu_\fP + o(X^{\ord(H)}),
\end{equation}
where $V_{0,\fP}$ denotes the completion of $V_0$ in $V(F_\fP)$ and $\ord(H) = \dim(V) / d$.

We can now include weight functions at finitely many finite primes by breaking $V_0$ up into a finite disjoint union of sublattices on which the weight functions are constant and apply \eqref{eq:countingmainre} to each sublattice. Adding in more and more finite primes gives the following upper bound.

\begin{theorem}\label{thm:countingupperbound}
 Suppose $G_0$ is a subgroup of $G(F)$ commensurable with $G(\cO)$ and $V_0$ is a $G_0$-invariant lattice in $V(F)$ commensurable with $V(\cO)$. Then
\begin{equation}\label{eq:countingmaininfupp}
 N_m(V_0,G_0,X) \leq \sqrt{D_F}^{\,-\dim V}\nu_{\infty}^{m_\infty}(\FF_0.R(X))\prod_{\fP\notin M_\infty}\int_{V_{0,\fP}}m_\fP(v)d\nu_\fP(v) + o(X^{\ord(H)}).
\end{equation}
\end{theorem}

\begin{proof}
 The proof is formal. See \cite[Theorem 2.21]{BS2}.
\end{proof}

To obtain the same lower bound, we need a uniformity estimate on the errors.

\vspace{10pt}
\noindent \textbf{AXIOM: Uniformity Estimate for $V$}: For any $\beta\in\cl(G)$ and any $\fP\notin M_\infty$, let $W_{\beta,\fP}$ denote the set of elements $v$ of $V_\beta$ such that $m_\fP(v)\neq 1$. Then for any positive real number $M,$
\begin{equation}
\label{eq:uniform}
 N\Big(\bigcup_{N\fP> M}W_{\beta,\fP},G_\beta,X\Big) = O\Big(\frac{X^{\ord(H)}}{f(M)}\Big) + o(X^{\ord(H)}),
\end{equation}
where $f(M)$ is a function of $M$ that approaches $\infty$ as $M$ approaches $\infty$.\hfill$\Box$
\vspace{10pt}

See Section \ref{sec:uniformity} for a list of methods one can use to check AXIOM: Uniformity Estimate for $V$. Once this tail estimate is obtained, the inequality in \eqref{eq:countingmaininfupp} may be turned into an equality:

\begin{theorem}\label{thm:countbeforevolume}
 For each $\beta$, let $\FF_\beta$ denote a fundamental domain for the action of $G_\beta$ on $G(F_\infty)$. Then
\begin{equation}\label{eq:countbeforevolume}
 N_{m_0}(V(F), G(F), X) = \sqrt{D_F}^{\,-\dim V}\sum_{\beta\in\cl(G)} \nu_\infty^{m_\infty}(\FF_\beta.R(X))\prod_{\fP\notin M_\infty}\int_{V_{\beta,\fP}}m_\fP(v)d\nu_\fP(v) + o(X^{\ord(H)}).
\end{equation}
\end{theorem}

It remains to compute the integrals at infinity and at $\fP\notin M_\infty$. Let $d\mu,d\tau$ denote left-invariant top differentials defined over $\cO$ on $S$ and $G$ respectively. Denote by $\mu_\infty,\mu_\fP,\tau_\infty,\tau_\fP$ the induced measures on $S(F_\infty),S(F_\fP),G(F_\infty),G(F_\fP)$, respectively. 
We now have the following change-of-measure formula (\cite[\S3.4]{BS2}). This requires AXIOM: \nicerep, and in particular that the sum of the degrees of the invariants equals the dimension of $V$.

\begin{proposition}\label{prop:changeofmeasure}
 There exists a nonzero constant $\J\in F^\times$ such that for any place $\fP$ of $F$, any open subset $R$ of $S(F_\fP)$, any continuous function $s:R\rightarrow V(F_\fP)$ such that $\inv\circ s = \mathrm{id}_R$ and any measurable function $\varphi$ on $V(F_\fP)$,  we have
\begin{equation}\label{eq:changeofmeasure}
 \int_{v\in G(F_\fP).s(R)}\varphi(v) d\nu_\fP(v) = |\J|_\fP\int_R\int_{G(F_\fP)} \varphi(gs(f))\,d\tau_\fP(g) d\mu_\fP(f),
\end{equation}
where $G(F_\fP).s(R)$ is viewed as a multiset and $|\J|_\fP$ denote the normalized $\fP$-adic valuation of $\J$ defined in \S\ref{sec:notation}.
\end{proposition}

To apply Proposition \ref{prop:changeofmeasure}, we need to know that the local sections $s$ exist.

\vspace{10pt}
\noindent \textbf{AXIOM: Local Spreading}: For any place $\fP$ of $F$ and any $v\in V(F_\fP)$ with $m_\fP(v)\neq 0$, there exists a $\fP$-adically open neighborhood $R$ of $\inv(v)$ and a continuous map $s:R\rightarrow V(F_\fP)$ such that $\inv\circ s = \textrm{id}_R$ and $s(\inv(v)) = v$. \hfill$\Box$
\vspace{10pt}

\begin{remark}
    We remark that if there is an algebraic section $s:S\rightarrow V$ defined over $\cO[1/N]$ for some $N\in \cO$, then AXIOM: Local Spreading is satisfied. 
\end{remark}

\begin{remark}\label{rem:RI}
    If $\{f\in S(F_\infty)\colon m_\infty(f)\neq0, H(f) \leq \lambda\}$ is compact (for example, if constructed by Proposition \ref{prop:Sigma}) and the number of $F_\infty$-orbits with invariants $f$ with $m_\infty(f)\neq 0$ is absolutely bounded, then AXIOM: Local Spreading gives the existence of the pre-compact fundamental domains $R_\lambda$ in AXIOM: Counting at Infinity I.
\end{remark}

For any place $\fP$ of $F$ and any $f\in S(F_\fP)$, we denote $\inv^{-1}(f)$ by $V_f(F_\fP).$ Then we have the following results.

\begin{lemma}\label{lem:volumecomputation}
 Fix some $\beta\in \cl(G)$ and $\fP\notin M_\infty$. Then
\begin{eqnarray}
\label{eq:massinf} \nu_\infty^{m_\infty}(\FF_\beta.R(X)) &=& |\J|_\infty\tau_\infty(\FF_\beta) \int_{S(F_\infty)_{X}}\sum_{v\in G(F_\infty)\backslash V_f(F_\infty)} \frac{m_{0,\infty}(v)}{\#\Stab_{G(F_\infty)}(v)}\,d\mu_\infty(f),\\
\nonumber\int_{V_{\beta,\fP}}m_\fP(v)d\nu_\fP(v)&=& \int_{V(\cO_\fP)}m_\fP(v)d\nu_\fP(v)\\
\label{eq:massfp}&=&|\J|_\fP\tau_\fP(G(\cO_\fP))\int_{S(F_\fP)} \sum_{v\in G(F_\fP)\backslash V_f(F_\fP)} \frac{m_{0,\fP}(v)}{\#\Stab_{G(F_\fP)}(v)}\,d\mu_\fP(f),
\end{eqnarray}
where $|\J|_\infty = \prod_{\fP\in M_\infty}|\J|_\fP$ and where $S(F_\infty)_X$ denotes the subset of $S(F_\infty)$ consisting of elements with height less than $X$ if $F$ is a number field, and equal to $X$ if $F$ is a function field.
\end{lemma}

\begin{proof}
 The first equation follows by covering $R(X)$ by the images of finitely many sections using the pre-compactness of $R(X)$ (AXIOM: Counting at Infinity I). The second equation follows from the $G(F_\fP)$-invariance of $m_\fP$ and $d\nu_\fP$; here the $G(F_\fP)$-invariance of $d\nu_\fP$ follows since semisimple groups have no nontrivial characters. The third equation follows from the same computation as in the proof of \cite[Corollary 3.8]{BS2}.
\end{proof}

Let $\mu_\infty^*$ be a normalization of $\mu_\infty$ such that the covolume of $S(\cO)$ in $S(F_\infty)$ is $1$. That is, $\mu_\infty^* = \sqrt{D_F}^{\,-\dim S}\mu_\infty$. Define
\begin{eqnarray*}
M_\infty(m_0,X) &=& \int_{S(F_\infty)_{X}}\sum_{v\in G(F_\infty)\backslash V_f(F_\infty)} \frac{m_{0,\infty}(v)}{\#\Stab_{G(F_\infty)}(v)}\,d\mu_\infty^*(f),\\
M_\fP(m_0) &=& \int_{S(F_\fP)} \sum_{v\in G(F_\fP)\backslash V_f(F_\fP)} \frac{m_{0,\fP}(v)}{\#\Stab_{G(F_\fP)}(v)}\,d\mu_\fP(f),\mbox{ for }\fP\notin M_\infty.
\end{eqnarray*}
We call these the \emph{local masses} at $\infty$ and $\fP$ respectively. Combining Theorem \ref{thm:countbeforevolume} and Lemma \ref{lem:volumecomputation} gives
\begin{equation}\label{thm:countingaftervolume}
 N_{m_0}(V(F),G(F),X) = \tau_{G,F}\, M_\infty(m_0,X)\!\prod_{\fP\notin M_\infty} M_\fP(m_0) + o(X^{\ord(H)}),
\end{equation}
where $\tau_{G,F}$ denotes the Tamagawa number of $G$ over $F$. Note here that we used $\dim \; S + \dim \; G = \dim \; V$ and absorbed $\sqrt{D_F}^{\,-\dim G}$ into the Tamagawa number $\tau_{G,F}$.

We use the same method (where $V=S$ and $G$ is the trivial group) to count the number of elements in $\Sigma=\text{inv}(V_P(F))$. The $G(F_\fP)$-invariance of $m_{0,\fP}$ for any $\fP$ implies that $\Sigma$ is defined by congruence conditions and so its characteristic function $\chi_\Sigma$ can be factored as $\prod_\fP \chi_{\Sigma,\fP}$. Moreover, Condition 3 of AXIOM: \localCondition~implies that $\Sigma\subset S(\cO)$. We define the local masses $M_\infty(\Sigma,X)$ and $M_\fP(\Sigma)$ for $\fP\notin M_\infty$ as follows:
\begin{equation}\label{eq:localmassF}
 M_\infty(\Sigma,X) = \int_{S(F_\infty)_{X}} \chi_{\Sigma,\infty} \,d\mu_\infty^*,\qquad
 M_\fP(\Sigma) = \int_{S(F_\fP)} \chi_{\Sigma,\fP} \,d\mu_\fP.
\end{equation}
We write $\ord(H)'$ for the order of magnitude of $M_\infty(\Sigma,X)$. Note that if $$1\ll\sum_{v\in G(F_\infty)\backslash V_f(F_\infty)} \frac{m_{0,\infty}(v)}{\#\Stab_{G(F_\infty)}(v)}\ll 1$$ as $f$ varies in $S(F_\infty)$ with $\chi_{\Sigma,\infty}(f)\neq0$, then $\ord(H)=\ord(H)'$.

\vspace{10pt}
\noindent \textbf{AXIOM: Uniformity Estimate for $S$}: For any $\fP\notin M_\infty$, let $W_{\fP}$ denote the set of elements $f$ of $S(\cO_\fP)$ such that $\chi_{\Sigma,\fP}(f)\neq1$. Then for any positive real number $M$,
\begin{equation}
\label{eq:uniformS}
 \#\Big\{f\in\bigcup_{N\fP> M}W_{\fP}:H(f)<X\Big\} = O_\epsilon\Big(\frac{X^{\ord(H)'}}{g(M)}\Big) + o(X^{\ord(H)'}),
\end{equation}
where $g(M)$ is a function of $M$ that approaches $\infty$ as $M$ approaches $\infty$, and where $X^{\ord(H)'}$ is the order of magnitude of the number of elements in $S(\cO_\fP)$ with height bounded by $X$.\hfill$\Box$
\vspace{10pt}

Then the number of elements of $\Sigma$ with height less than $X$ if $F$ is a number field and equal to $X$ if $F$ is a function field is
\begin{equation}\label{eq:countforS}
 M_\infty(\Sigma,X)\prod_{\fP\notin M_\infty}M_\fP(\Sigma) + o(X^{\ord(H)'}).
\end{equation}

We can now state our main theorem.

\begin{theorem}\label{thm:countingmain}
 Let $F$ be a global field and let $M_\infty$ be the set of infinite places of $F$ when $F$ is a number field and be a finite set of places corresponding to a finite set of closed points of $\sc$ when $F$ is the function field of $\sc$. Let $(G,V)$ be a representation over the ring of integers $\cO$ with GIT quotient $S$ and let $\inv:V\rightarrow S$ denote the natural invariant map. Let $V(F)^\gen$ be a $G(F)$-invariant subset of $V(F)$. Let $m_0$ be a weight function on $V(F)$ defined by congruence conditions. Let $V_P(F)$ denote the subset of $V(F)$ where $m_0$ is nonzero. Write $\Sigma=\inv(V_P(F))$ and let $H$ denote a height function on $\Sigma$. Suppose the following axioms are satisfied:
\begin{itemize}
 \item[{\rm 1.}] AXIOM: \nicerep,
 \item[{\rm 2.}] AXIOM: \localCondition,
 \item[{\rm 3.}] AXIOM: Counting at Infinity I and II,
 \item[{\rm 4.}] AXIOM: Local Spreading,
 \item[{\rm 5.}] AXIOM: Uniformity Estimate for $S$.
\end{itemize}
Then the average number of $G(F)$-orbits in $V(F)^\gen$ with height at most $X$ if $F$ is a number field and equal to $X$ if $F$ is a function field, where each $G(F)v$ orbit is weighted by $m_0(v)/\#\Stab_{G(F)}(v)$, is bounded above by
\begin{equation}\label{eq:average}
 \tau_{G,F}\,\frac{M_\infty(m_0,X)\prod_{\fP\notin M_\infty} M_\fP(m_0)}{M_\infty(\Sigma,X)\prod_{\fP\notin M_\infty} M_\fP(\Sigma)}.
\end{equation}
If moreover AXIOM: Uniformity Estimate for $V$ also holds, then this upper bound is also a lower bound.
\end{theorem}

For convenience, we group the first three axioms in Theorem \ref{thm:countingmain} together and refer to them as \emph{AXIOM: Counting}.

\section{Verification of axioms over global fields}\label{sec:count}

In this section, we describe how to generalize the verification of the counting and uniformity axioms from the case of $F=\Q$ to general global fields $F$.

\subsection{Verifying the counting axioms}\label{sec:transference}

The counting axiom that typically presents the greatest difficulty is AXIOM: Counting at Infinity II. In the many cases in the literature where this axiom has been verified for $F=\Q$, it has typically been achieved by proving specific combinatorial ``cutting off the cusp'' results.
In this subsection, we prove that this combinatorial input in fact suffices to verify the axiom for general global fields.

Let $P$ be a minimal parabolic $F$-subgroup of $G$ and let $T$ be a maximal split torus of $G$ contained in $P$. Let $\Delta$ denote a basis of positive roots defined by $P$. The non-compactness of $\FF_0$ is equivalent to $G$ being not anisotropic (\cite[Corollary 12.4]{BHC}), i.e., $T$ is non-trivial. The restriction of the representation $V$ of $G$ to the torus $T$ breaks up into a direct sum of characters. Let $U_0$ denote the set of all occurring characters (with multiplicity) so that as a representation of $T$, we have $$V\cong\bigoplus_{\chi\in U_0}\chi.$$
For any $v\in V$ and any $\chi\in U_0$, we denote by $v(\chi)$ the $\chi$-isotypical component of $v$. For any subset $U$ of $U_0$, we say that an element $v\in V(F)$ is $U$-\emph{non-generic} if there exists $g\in G(F)$ such that $(gv)(\chi) = 0$  for all $\chi\in U$. 
Suppose there are subsets $U_1,\ldots,U_n$ of $U$ such that if $v\in V(F)$ is $U_i$-non-generic for some $i=1,\ldots,n$, then $v$ is non-generic.

For any set $U$ of characters of $T$, we write $\Q[U]$ for the group algebra of $U$. An element $\sum_{\chi\in U}a_\chi\,\chi\in\Q[U]$ is in $\Q^{\geq0}[U]$ (resp. $\Q^{\leq0}[U]$, $\Q^{<0}[U]$) if $a_\chi\geq0$ (resp. $a_\chi\leq0$, $a_\chi<0$) for all $\chi\in U.$ For any $\pi=\sum_{\chi\in U_0}a_\chi\,\chi\in\Q[U_0]$, let $\deg(\pi)=\sum_{\chi\in U_0} a_\chi$ denote the degree of $\pi$. We follow the convention that $\deg 0 = -\infty$. 

Since $G$ is semisimple, every character of $T$ is inside $\Q[\Delta]$. For $\pi_1,\pi_2\in\Q[\Delta]$, we say $\pi_1\geq\pi_2$ if $\pi_1-\pi_2\in\Q^{\geq0}[\Delta].$ For any set $U$ of characters of $T$ and $\pi\in\Q[\Delta]$, we write $U\geq \pi$ if there exists $\pi'\in\Q^{\geq0}[U]$ with $\deg(\pi')=1$ such that $\pi'\geq\pi$. If $U,U'$ are two sets of characters of $T$, we write $U\geq U'$ if $U\geq \chi$ for every $\chi\in U'$. Otherwise, we write $U\not\geq U'$. A priori we have $U\geq U'$ if $U\supset U'$. We say a subset $U\subset U_0$ is \emph{saturated} if $U$ contains all the sets $U'\subset U_0$ such that $U\geq U'$. 

We prove the following result.

\begin{theorem}\label{prop:maintransference}
Suppose the following conditions are satisfied:
\begin{itemize}
\item[{\rm 1.}] $\prod_{\fP\notin M_\infty}\lambda_\fP = 0$ where $\lambda_\fP$ denotes the $\fP$-adic density of non-generic elements in $V(\cO)$ for every $\fP\notin M_\infty$;
\item[{\rm 2.}] for any saturated subset $U\subset U_0$ that does not contain $U_i$ for any $i=1,\ldots,n$, there exists some $\pi_U\in \Q^{\geq0}[U_0-U]$ with $\deg(\pi_U)<\#U$ such that $$\pi_U+\delta+\sum_{\chi\in U_0-U}\chi \in \Q^{<0}[\Delta];$$
\item[{\rm 3.}] for any $\alpha\in\Delta$, there exists some $\pi_\alpha\in\Q^{\leq0}[U_0]$ such that $\alpha\leq\pi_\alpha$.
\end{itemize}
Then AXIOM: Counting at Infinity II is satisfied.
\end{theorem}

\begin{remark}
    We note also that all three conditions in Theorem \ref{prop:maintransference} are independent of the choice of the height function. Moreover, if the representation $(G,V)$ over $\cO$ is base changed from $\Z$ and the group $G$ is split over $\Z$, then Conditions 2 and 3 do not depend on the field $F$. If in addition $F$ is a number field, then Condition 1 follows if one proves that there is a positive absolute constant $c$ such that for a density $1$ set of primes $p$, the $p$-adic density $\lambda_p$ of non-generic elements in $V(\Z)$ satisfies $\lambda_p < 1 - c$, for the same bound is then true for any degree $1$ prime $\fP$ of $F$ over $p$.
\end{remark}

\begin{remark}
    If we change Condition 1 to $\lambda_\fP < 1 - c$ for some absolute constant $c$, for all but finitely many $\fP$, then a power saving error term can be obtained, e.g., using the large sieve (see \cite[Proposition 4.3]{BSWsqf2}).
\end{remark}

We first recall some results from reduction theory (\cite{Sp}). For any positive constants $c$ and $c'$, define:
\begin{eqnarray*}
 T(c) &=& \{t=(t_w)_{w\in M_\infty}\in T(F_\infty) : |\alpha(t)| \geq c,\forall \alpha\in \Delta\},\\
 T(c,c') &=& \{t\in T(c) : \frac{|\alpha(t_w)|_w^{1/\deg w}}{|\alpha(t_{w'})|_{w'}^{1/\deg w'}} \leq c',\forall\alpha\in\Delta,\forall w,w'\in M_\infty\}.
\end{eqnarray*}
Then by \cite[Remark 2.2]{Sp} there exists positive real numbers $c,c'$, a compact subset $N$ of the group of $\F_\infty$-points of the unipotent radical $U_P$ of $P$, and a compact subgroup $K$ of $G(F_\infty)$ such that $$G(F_\infty) = G(\cO)NT(c,c')K.$$
We remark that in \cite{Sp}, the set $T(c)$ is defined by $|\alpha(t)| \leq c.$ Since we want a fundamental domain for the left action of $G(\cO)$ on $G(F_\infty)$, we need to apply inverses to the results of \cite{Sp}. The extra parameter $c'$ comes from computing $T(\cO)\backslash T(c)$ using the fact that the image $\Lambda$ of $\cO^\times$ in $\R^{|M_\infty|}$ under the map $t\mapsto (\log |t_w|_w)$ is a lattice of rank $|M_\infty|-1$ in the hyperplane $H:x_1+\cdots+x_v = 0$ with compact quotient $H/\Lambda$. The factors $1/\deg w$ and $1/\deg w'$ are not necessary if one does not need the optimal power-saving error term; see Remark \ref{rem:S1}.

Since $G_0$ is commensurable with $G(\cO)$, we can construct a fundamental domain $\FF_0$ for the action of $G_0$ on $G(F_\infty)$ that is contained in the finite union
\begin{equation}\label{eq:fundamental}
 \FF_0 \subset \bigcup_{j=1}^l \alpha_j N T(c,c') K,
\end{equation}
where $\alpha_j\in G(F)$.

By writing $V_0$ as a finite disjoint union of subsets over which the weight function $m_\infty$ is constant, it suffices to verify AXIOM: Counting at Infinity II when $m_\infty=1$. Moreover, by AXIOM: \localCondition, we may further assume that $\#\Stab_{G(F_\infty)}(v) = m$ is constant for $v\in V_0$. Fix some open ball $C_0\subset G(F_\infty)$ of finite nonzero measure. Then we have
\begin{eqnarray*}
 N_{m_\infty}(V_0,G_0,X) &=& \frac{1}{m}\#\{\FF_0.R(X) \cap V_0^\gen\}\\
&=& \frac{1}{m\tau_\infty(C_0)}\int_{C_0} \#\{\FF_0g.R(X) \cap V_0^\gen\}\,d\tau_\infty(g)\\
&=& \frac{1}{m\tau_\infty(C_0)}\int_{\FF_0} \#\{hC_0R(X) \cap V_0^\gen\}\,d\tau_\infty(h),
\end{eqnarray*}
where the last line follows as in \cite[Theorem 2.5]{BS2}. Without loss of generality, we assume that $l=1$ in the finite union \eqref{eq:fundamental}. That is, $\alpha^{-1}\FF_0=\FF\subset N T(c,c') K'$ for some $\alpha\in G(F)$. Using the left-invariance of the Haar measure, we have
$$N_{m_\infty}(V_0,G_0,X) = \frac{1}{m\tau_\infty(C_0)}\int_{\FF} \#\{hC_0R(X) \cap \alpha^{-1}V_0^\gen\}\,d\tau_\infty(h)$$
Since $T(c,c')$ is non-compact, we break up the above integral into two pieces: the main body and the cusp.

Recall that $R(X) = [1,X^{1/(d[F:\Q])}].R_1$ when $F$ is a number field and $R(X)=\Lambda_XR_\lambda$ when $F$ is a function field, where $\Lambda_X\in F$ and $\lambda$ belongs to a fixed finite set. We set $\Lambda_X = X^{1/(d[F:\Q])}$ when $F$ is a number field. The Haar measure $d\tau_\infty$ is given by $$d\tau_\infty =|\delta(t)|\, dud^\times tdk$$ up to a fixed constant, where $\delta \in \Q^{<0}[\Delta]$ is obtained from the action of $T$ on the unipotent radical $U_P$ of $P$, and $dk,d^\times t,du$ are Haar measures on $K,T(F_\infty),U_P(F_\infty)$, respectively. Note that if $t\in T(c)$ and $N\subset U_P(F_\infty)$ is compact with absolutely bounded coordinates, then so is $t^{-1}Nt$. That is, there exists a compact set $N'\subset U_P(F_\infty)$ independent of $t$ containing all of $t^{-1}Nt$ as $t$ varies in $T(c)$. Let $C_1$ denote the compact set $N'K'C_0.$ Then there is a constant $c_1>0$ depending on $C_1$ such that for any $v\in tC_1R(X)$ and $\chi\in U_0$, we have
$$|v(\chi)| \leq c_1 |\Lambda_X||\chi(t)|.$$

We explain the meaning of saturated subsets (defined just above Theorem \ref{prop:maintransference}). For any $\pi=\sum_{\chi\in U_0}a_\chi\,\chi\in\Q[U_0]$, we write
$$|\pi(t)| = \prod_{\chi\in U_0} |\chi(t)|^{a_\chi}.$$ 
Suppose $U$ is a subset of $U_0$. Suppose $Y>0$ is a real number and $t\in T(c,c')$ such that $|\chi(t)| < Y$ for all $\chi\in U$. Then for any $\pi\in \Q[\Delta]$ with $U\geq\pi$, we have $$|\pi(t)| \ll_{c,U} Y.$$

Since $V_0$ is commensurable with $V(\cO)$, there exists a positive constant $\kappa$ such that for any $v\in \alpha^{-1}V_0$ and $\chi\in U_0$, we have either $v(\chi) = 0$ or $|v(\chi)| > \kappa.$ For any subset $U$ of $U_0$, define $V(F_\infty)(U)$ by $$V(F_\infty)(U) = \{v\in V(F_\infty):|v(\chi)| < \kappa,\forall \chi\in U\}.$$
For any saturated subset $U$ of $U_0$, let $T_U$ be the subset of $T(c,c')$ consisting of $t$ such that:
\begin{enumerate}
    \item for any $\chi\in U_0$ such that $U\geq \chi$, we have $|v(\chi)| < \kappa$ for any $v\in tC_1R(X)$;
    \item for any $\chi\in U_0$ such that $U\not\geq \chi$, we have $|v(\chi)| \geq \kappa$ for some $v\in tC_1R(X)$.
\end{enumerate}
In light of the above, we see that if $t\in T_U$, then
$$|\Lambda_X||\chi(t)| <\kappa c_1^{-1} \mbox{ for all }\chi\in U\mbox{ and }|\Lambda_X||\chi(t)|\geq \kappa c_1^{-1}\mbox{ for all }\chi\notin U.$$
We define $\FF_U = \FF\cap NT_UK'$. The point here is that if $h\in \FF_U$, then for any $v\in hG_0R(X) \cap \alpha^{-1}V_0$ and any $\chi\in U$, we have $v(\chi) = 0$. As a consequence, if $U\supset U_i$ for some $i = 1,\ldots,n$, then any $v\in hG_0R(X) \cap \alpha^{-1}V_0$ is non-generic.

\begin{remark}
    Either one of $V(F_\infty)(U)$, $T_U$, or $\FF_U$ can be referred to as the \emph{cusp associated to $U$}, with $V(F_\infty)(\emptyset)$, $T_\emptyset$, or $\FF_\emptyset$ being the \emph{main body}.
\end{remark}

We now estimate the contribution from $$\int_{\FF_U} \#\{hC_0R(X) \cap \alpha^{-1}V_0^\gen\}\,d\tau_\infty(h)$$
for a saturated subset $U$. Consider first $U = \emptyset$. We show first that $T_\emptyset$ is pre-compact. By the definition of $T(c,c')$, it suffices to give an upper bound for $|\alpha(t)|$ for each $\alpha\in\Delta$ and every $t\in T_\emptyset.$ Fix any $t\in T_\emptyset$ and $\alpha\in\Delta$. Condition 3 gives some $\pi_\alpha\in\Q^{\leq0}[U_0]$ such that $\alpha\leq\pi_\alpha$. Then $$|\alpha(t)| \ll |\pi_\alpha(t)| \ll |\Lambda_X|^{-\deg(\pi_\alpha)}$$
where the implied constant depends on $\alpha$ and $\pi_\alpha$. Therefore, if for any $\alpha\in\Delta$, there exists some $\pi_\alpha\in\Q^{\leq0}[U_0]$ such that $\alpha\leq\pi_\alpha$, then we have an upper bound for $|\alpha(t)|$ for every $\alpha\in\Delta$.

We can now apply Davenport's lemma (Proposition \ref{prop:Davenport}) to estimate $\#\{hC_0R(X) \cap \alpha^{-1}V_0\}$ for $h\in\FF_\emptyset$. In this region, since no coordinate is forced to be small, we have
\begin{eqnarray*}
 \int_{\FF_\emptyset} \#\{hC_0R(X) \cap \alpha^{-1}V_0\}\,d\tau_\infty(h) &=&\int_{\FF_\emptyset} \Vol_{\alpha^{-1}V_0}\{hC_0R(X)\}\,d\tau_\infty(h)+ o(X^{\ord(H)})\\
&=& \int_{\FF_\emptyset} \Vol_{V_0}\{hC_0R(X)\}\,d\tau_\infty(h)+ o(X^{\ord(H)})
\end{eqnarray*}
where the last equality follows since the Haar measure on $V(F_\infty)$ is $G(F_\infty)$-invariant. Here we use $\Vol_\L$, for a lattice $\L$ commensurable with $V(\cO)$, to denote the Euclidean volume normalized so that the covolume of $\L$ is $1$, and we recall that $\Vol_{V_0} = \nu_{\infty,0}$. Note here, 
$$\Vol_{V_0}(hC_0R(X)) \asymp \prod_{\chi\in U_0}|\Lambda_X||\chi(t)| = |\Lambda_X|^{\#\dim V} \asymp X^{\ord(H)},$$
using the $G(F_\infty)$-invariance of the measure on $V(F_\infty)$  to conclude that $\prod_{\chi\in U_0}|\chi(t)| = 1$. 

By Condition 2, we know the non-generic points are negligible in the main body:
$$\int_{\FF_\emptyset} \#\{hC_0R(X) \cap \alpha^{-1}V_0\}\,d\tau_\infty(h) = \int_{\FF_\emptyset} \#\{hC_0R(X) \cap \alpha^{-1}V_0^\gen\}\,d\tau_\infty(h) + o(X^{\ord(H)}).$$
Hence, in order to conclude that
$$\int_{\FF_\emptyset} \#\{hC_0R(X) \cap \alpha^{-1}V_0^\gen\}\,d\tau_\infty(h) = \int_{\FF} \Vol_{V_0}\{hC_0R(X)\}\,d\tau_\infty(h)+ o(X^{\ord(H)}),$$
it remains to show that the cuspidal region has negligible volume. This amounts to proving that $$\int_{t\in T(c,c')-T_\emptyset} |\Lambda_X|^{\dim V}|\delta(t)|d^\times t = o(|\Lambda_X|^{\dim V}).$$
Since $\delta\in \Q^{<0}[\Delta]$, it suffices to show that there is a positive real number $c_2$ such that for any $t\in T(c,c')-T_\emptyset$, there is some $\alpha\in\Delta$ with $|\alpha(t)| > X^{c_2}$. 
Let $\{\chi_1,\ldots,\chi_m\}\subset U_0$ be a complete set of minimal weights. Note $m>1$ if and only if the representation $V$ of $G$ is non-generic. 
Fix any $t\in T(c,c')-T_\emptyset$. Then $|\Lambda_X||\chi_i(t)|<c_3$ for some absolute constant $c_3$ for some $i=1,\ldots,m$. In particular, $\chi_i$ is nontrivial. Since $\chi_i$ is a minimal weight, we have that $\chi_i\in\Q^{\leq0}[\Delta].$ Hence a lower bound of $|\chi_i^{-1}(t)|$ of the form $c_3^{-1}|\Lambda_X|$ gives a lower bound of $|\alpha(t)|$, for some $\alpha\in \Delta$, of the form $(c_3^{-1}|\Lambda_X|)^{-1/\deg(\chi_i)}$.

We now deal with $U$ non-empty. For $t\in T_U$, Davenport's lemma (Proposition \ref{prop:Davenport}) gives
$$\#tC_1R(X) \cap \alpha^{-1}V_0 = O\Big(\prod_{\chi\in U_0-U}|\Lambda_X||\chi(t)|\Big).$$
Hence, it suffices to bound
$$I(U,X) = |\Lambda_X|^{\#U_0 - \#U}\int_{T_U}\Big(\prod_{\chi\in U_0-U}|\chi(t)|\Big)|\delta(t)|\,d^\times t.$$
By Condition 2, there exists $\pi_U \in \Q^{\geq 0}[U_0-U]$ with $\deg(\pi_U) < \#U$ such that 
$$\pi_U+\delta+\sum_{\chi\in U_0-U}\chi \in \Q^{<0}[\Delta].$$
Then for any $t\in T_U$, we have
$$|\Lambda_X|^{\deg\pi_U}|\pi_U(t)| \gg 1\qquad\mbox{and}\qquad |\pi_U(t)|\Big(\prod_{\chi\in U_0-U}|\chi(t)|\Big)|\delta(t)|\ll 1.$$
Combining these two bounds gives
$$I(U,X) = O\big(|\Lambda_X|^{\#U_0 - \#U + \deg \pi_U}\big) = o(|\Lambda_X|^{\dim V}).$$
This completes the proof of Theorem \ref{prop:maintransference}.

\subsection{Verifying the uniformity axioms}\label{sec:uniformity}

In this section, we describe three methods that have been previously used to prove AXIOM: Uniformity Estimate for $S$ and for $V$, in the case $F=\Q$. We prove that these extend to the case of global fields without much change. The results of this section will suffice to prove all the uniformity estimates needed for our main theorems, with two exceptions. AXIOM: Uniformity Estimates for $S$ and for $V$ for the $2$-Selmer groups of elliptic curves are proved by hand in \S5. Furthermore, AXIOM: Uniformity Estimate for $S$ for monic degree-$n$ polynomials requires ``counting in the cusp'' techniques, and are proved by Oller \cite{Oller}.  

We begin by restating the axiom in sufficient generality to cover both AXIOM: Uniformity Estimate for $S$ and for $V$.
Let $G$ be an affine algebraic group with a representation $V$ defined over $\cO$. We refer to Section \ref{sec:general} for the definitions of the class group $\cl(G)$ and $V_\beta,G_\beta$ for any $\beta\in\cl(G)$. Let $H$ be a height function on $V(F)$. For any subgroup $G_0$ of $G(F)$ commensurable with $G(\cO)$ and any $G_0$-invariant subset $V_0$ of $V(F)$ commensurable with $V(\cO)$, let $N(V_0,G_0,X)$ denote a counting function for the number of $G_0$-orbits in $V_0$ with height bounded by $X$. For AXIOM: Uniformity Estimate for $V$, $N(V_0,G_0,X)$ will count only the generic orbits. 
For AXIOM: Uniformity Estimates for $S$, we may take $V$ to be $S$ and $G$ to be trivial, and where $N(V_0,G_0,X)$ counts all orbits. 
Define the order $\ord(H)$ of $H$ to be the infimum of real numbers $d$ such that $N(V_0,G_0,X)=O(X^d).$ Note this definition is independent of $G_0$ and $V_0$.

Let $B_\fP\subset V(\cO_\fP)$ denote a set of elements that we want to sieve out. Define
$$W_{\beta,\fP} = \{v\in V_\beta \colon \beta_\fP v  \in B_\fP\}.$$
We want to obtain the following uniformity estimate, namely, AXIOM: Uniformity Estimate for $V$:
%
For any positive real number $M,$ we have
\begin{equation}
\label{eq:uniform1}
 \sum_\beta N\Bigl(\bigcup_{N\fP> M}W_{\beta,\fP},G_\beta,X\Bigr) = O\Big(\frac{X^{\ord(H)}}{f(M)}\Big) + o(X^{\ord(H)}),
\end{equation}
where $f(M)$ is a function of $M$ that approaches $\infty$ as $M$ approaches $\infty$. If this holds, we will also say AXIOM: Uniformity Estimate is satisfied for $(V,G,\{B_\fP\},N).$
\hfill$\Box$\vspace{2 ex}

\subsubsection*{The geometric sieve}

We view $V$ as the affine scheme $\A_{\cO}^{\dim(V)}$ over $\cO$. Suppose there is a closed subscheme $Y$ of $V$ of codimension $k\geq 2$ such that for any finite prime $\fP$, any element of $B_\fP$ reduces modulo $\fP$ to an element of $Y(k(\fP))$, where $k(\fP)$ denotes the residue field of $\fP.$ In this case, we can use the techniques of \cite{Bgeosieve}. The following two results are the natural generalizations to global fields of their counterparts in \cite{Bgeosieve} where we replace the use of Davenport's result on the number of lattice points in some compact region in $\R^n$ by its function field analogue if $F$ is a function field.

\begin{lemma}\label{lem:Lem5F}
 Let $B$ be a compact region in $F_\infty^n$ and let $r\in \R$ if $F$ is a number field and $r\in F$ if $F$ is a function field. Let $Y$ be a closed subscheme of $\A_{\cO}^n$ of codimension $k\geq 0.$ Then we have,
\begin{equation}
 \#(rB\cap Y(F_\infty) \cap \cO^n) = O(|r|^{n-k}),
\end{equation}
where the implied constant depends only on $B$, $Y$ and $F$.
\end{lemma}

\begin{proof}
 The same proof of \cite[Lemma 5]{Bgeosieve} carries over. The only difference is that $\#(rB \cap \cO^n) = O(|r|^n).$
\end{proof}

\begin{theorem}\label{thm:Th7F}
 Let $B$ be a compact region in $F_\infty^n$ and let $r\in \R$ if $F$ is a number field and $r\in F$ if $F$ is a function field. Let $Y$ be a closed subscheme of $\A_{\cO}^n$ of codimension $k\geq 2.$ Let $M$ be a positive real number. Then we have,
\begin{eqnarray}
 \nonumber&&\#\{a\in rB\cap \cO^n: a \,(\mathrm{mod}\,\fP) \in Y(k(\fP))\mbox{ for some prime }\fP\notin M_\infty\mbox{ with }N\fP>M\} \\
 \label{eq:estF}&=& O\left(\frac{|r|^n}{M^{k-1}\log M} + |r|^{n-k+1}\right),
\end{eqnarray}
where the implied constant depends only on $B$, $Y$ and $F$.
\end{theorem}

\begin{proof}
 The estimate of the error term coming from primes $\fP$ with $N\fP>|r|$ follows as in the proof of \cite[Theorem 3.3]{Bgeosieve} with \cite[Lemma 3.1]{Bgeosieve} replaced by Lemma \ref{lem:Lem5F} above. The main term works the same way as well using Proposition \ref{prop:cong}. 
\end{proof}

\begin{theorem}\label{thm:geosieve}
 Suppose there is a closed subscheme $Y$ of $V$ of codimension $k\geq 2$ such that for all but finitely many finite primes $\fP$, any element of $B_\fP$ reduces modulo $\fP$ to an element of $Y(k(\fP))$. Assuming \emph{AXIOM: Counting}, then
\begin{equation}\label{eq:geosieve}
 \sum_\beta N(\bigcup_{N\fP> M}W_{\beta,\fP},G_\beta,X) = O\big(\frac{X^{\ord(H)}}{M^{k-1}\log M}\big) + o(X^{\ord(H)}).
\end{equation}
In particular, the desired uniformity estimate \eqref{eq:uniform1} holds.
\end{theorem}

We will give a generalization (Theorem \ref{thm:redsieve}) of Theorem \ref{thm:geosieve} in the next section. See \cite[\S3.4]{Bgeosieve} for the proof of Theorem \ref{thm:redsieve}.

\subsubsection*{The reduction sieve}

In this section, we present a method of checking AXIOM: Uniformity Estimate via the use a bigger group $G'$ to reduce the height of elements in $W_{\beta,\fP}.$

We now state the main result of this section, which reduces to Theorem \ref{thm:geosieve} when $G'=G$.

\begin{theorem}\label{thm:redsieve}
 Suppose there exists an affine algebraic group $G'$ over $\cO$ containing $G$ as a subgroup. Suppose $G'$ acts on $V$ such that its restriction to $G$ coincides with the original action of $G$ on $V$ and that $G'(F_\fP)$ preserves $V_P(F_\fP)$. Suppose the following conditions hold:
\begin{enumerate}
\item There exists an algebraic character $\delta:G'\rightarrow\bG_m$ that is trivial on $G$ and such that $|\delta(G'(\cO_\fP))|_\fP~=~1$ for any $($finite$)$ prime $\fP$ and that for any $g\in G'(F)$ and any $v\in V(F)$, $$H(gv) = \prod_{\fP\notin M_\infty}|\delta(g)|_\fP\, H(v).$$
 \item There exists $G$-invariant closed subschemes $Y_j$ of $V$ over $\cO$ of codimension $j$ for each $j=0,\ldots,\dim V$ and a positive constant $\eta$ such that for any $\fP\notin M_\infty$ and any $w_\fP\in B_\fP$, there exists a positive real number $\alpha(w_\fP)$, an element $g_\fP(w_\fP)\in G'(F_\fP)$ and an integer $k(w_\fP)=0,\ldots,\dim V$ such that
\begin{enumerate}
\item $g_\fP(w_\fP) w_\fP\in V(\cO_\fP)$;
\item the reduction of $g_\fP(w_\fP) w_\fP$ modulo $\fP$ is in $Y_{k(w_\fP)}(k(\fP))$;
\item $\ord(H)\alpha(w_\fP) + k(w_\fP) - 1\geq\eta$;
\item $|\delta(g_\fP)|_\fP\leq (N\fP)^{-a(w_\fP)}$.
\end{enumerate}
\item The fibers of the map $\pi_{\beta,\fP}$ induced by the above data as defined in \eqref{eq:reduction} below have sizes bounded above by some absolute constant $C_2$.
\end{enumerate}
Then assuming AXIOM: Counting, we have
\begin{equation}\label{eq:reductionsieve}
 \sum_\beta N(\bigcup_{N\fP> M}W_{\beta,\fP},G_\beta,X) = O\big(\frac{X^{\ord(H)}}{M^\eta\log M}\big) + o(X^{\ord(H)}).
\end{equation}
In particular, the desired uniformity estimate \eqref{eq:uniform1} holds.
\end{theorem}

We now give the definition for the map $\pi_{\beta,\fP}$. Take any $w_\fP\in W_{\beta,\fP}$. Then, by definition $\beta_\fP v  \in B_\fP$ and $g_\fP(\beta_\fP v)\beta_\fP v\in V(\cO_\fP)$. Consider the following element $g'\in G'(\A_f)$ defined by $g'_\fP=g_\fP(\beta_\fP v) \beta_\fP$ and $g'_{\fP'}=\beta_{\fP'}$ for all $\fP'\neq\fP$. Decompose $G'(\A_f)$ into a finite disjoint union of double cosets: $$G'(\A_f) = \bigcup_{\gamma\in\cl(G')}\left(\prod_{\fP'\notin M_\infty}G'(\cO_{\fP'})\right)\gamma \,G'(F).$$
Fix some representative of each $\gamma\in\cl(G')$ in $G'(\A_f)$. Then there exists $\gamma\in\cl(G')$, $h'\in G'(F)$, and $h_{\fP'}\in G'(\cO_{\fP'})$ for any $\fP'\notin M_\infty$ such that $g'=(h_{\fP'})_{\fP'\notin M_\infty}\gamma h'$ in $G'(\A_f)$. Set $v'\in h'v\in V(F)$. Then $v'$ lies in
 $$V_\gamma := V(F) \cap \gamma^{-1}\left(\prod_{\fP'\notin M_\infty}V(\cO_{\fP'})\right),$$
 with height $$H(v')\leq C_1\frac{H(v)}{(N\fP)^{\alpha(\beta_\fP v)}}$$
 for some constant $C_1=\prod_{\fP'\notin M_\infty} |\delta(\gamma_{\fP'})|_{\fP'}$ depending only on $G$, $G'$ and $F$. Therefore, we set $\pi_{\beta,\fP}(v)=v'$ and have just defined the following map:
\begin{equation}\label{eq:reduction}
 \pi_{\beta,\fP}: G_\beta\backslash W_{\beta,\fP} \rightarrow \bigcup_{\gamma\in \cl(G')}G'_\gamma\backslash V_\gamma.
\end{equation}

See \cite[\S3.4]{Bgeosieve} (particularly \cite[Lemma 3.7]{Bgeosieve}) for the proof of Theorem \ref{thm:redsieve}. We remark that the entirety of AXIOM: Counting at Infinity II is not necessary for Theorems \ref{thm:geosieve} and \ref{thm:redsieve}. All that is needed is Condition 3 of Theorem \ref{prop:maintransference} which implies that the main body of the fundamental domain is compact.

\subsubsection*{The embedding sieve}\label{sec:embed}

In certain situations, the original representation $V$ is not big enough for one to exhibit a larger group $G'$ and to apply the reduction sieve. In this section, we describe an approach where one embeds $V$ into a larger space where the uniformity estimate is already known.

\begin{theorem}\label{thm:embsieve}
 Suppose $G'$ is an affine algebraic group over $\cO$ with a representation $V'$. For any $\gamma\in\cl(G')$, let $V'_\gamma$ and $G'_\gamma$ be defined as in \eqref{eq:V_beta} and \eqref{eq:G_beta}. Let $H'$ be a homogeneous height function on $V'(F_\infty)$ and let $N'$ be a counting function for $V'$. Let $\pi:V\rightarrow V'$ denote an algebraic map over $\Spec\cO$ that descends to a map on orbits (on field-valued points). Then for any $\beta\in\cl(G)$ and any $v\in V_\beta$, there exists some $\gamma\in\cl(G')$ such that $\pi(v)\in V'_\gamma$. Moreover, suppose the following conditions hold:
\begin{enumerate}
 \item The fibers of the maps $\pi_\beta$, for any $\beta\in\cl(G)$,
\begin{equation}\label{eq:pibeta}
 \pi_\beta:G_\beta\backslash V_\beta \rightarrow \bigcup_{\gamma\in\cl(G')}G'_\gamma\backslash V'_\gamma.
\end{equation}
induced by $\pi$ have absolutely bounded sizes.
 \item There exists a constant $\alpha>0$ satisfying $\ord(H')\alpha \leq \ord(H)$ such that for any $\beta\in\cl(G)$, all but finitely many finite prime $\fP$ and any $v\in W_{\beta,\fP}$,
\begin{enumerate}
\item $H'(\pi(v)) \leq C_1 H(v)^\alpha,$ for some absolute constant $C_1$;
\item the orbit of $\pi(v)$ is counted in $N'$ whenever the orbit of $v$ is counted in $N$.
\end{enumerate}
 \item For all but finitely many finite primes $\fP$, there exists a subset $B'_\fP\subset V'(\cO_\fP)$ containing $\pi(B_\fP)$ such that \emph{AXIOM: Uniformity Estimate} is satisfied for $(V',G',\{B'_{\fP}\},N')$.
\end{enumerate}
Then \emph{AXIOM: Uniformity Estimate} is satisfied for $(V,G,\{B_{\fP}\},N)$.
\end{theorem}
The proof is immediate.


\section{The average size of the $n$-Selmer groups of elliptic curves}\label{sec:elliptic}

In this section, we extend the results on the average sizes of the $n$-Selmer groups of elliptic curves, for $n=2,3,4,5$, to global fields. We verify the axioms listed in Theorem \ref{thm:countingmain} and prove Theorem \ref{ecthm}.

Let $F$ be a global field of characteristic not equal to $2$ or $3$. An elliptic curve $E/F$ can be expressed as $$E = E_{A,B}:y^2 = x^3 + Ax + B,$$ where $(A,B)\in F$. Two elliptic curves $E_{A,B}$ and $E_{A',B'}$ are equivalent if and only if there is some $\alpha\in F$ such that $A'=\alpha^4A$, $B'=\alpha^6B$. Hence it is natural to view $(A,B)$ as an element of the weighted projective space $\P(4,6).$ We define the height of $E_{A,B}$, and the height of $(A,B)$, by the usual height on weighted projective space defined in Section \ref{sec:heightweighted}. That is, let $I$ be the ideal
$$I = \{a\in F:a^4A\in \cO,a^6B\in\cO\},$$
then,
\begin{equation}\label{eq:heightweightelliptic}
H(A,B) = (NI)\prod_{\fP\in M_\infty} \max(|A|_\fP^{1/4},|B|_\fP^{1/6}).
\end{equation}

Let $S=\A^2$ be the space of pairs $(A,B)$ with an action of $\bG_m$ given by $\alpha.(A,B)=(\alpha^4A,\alpha^6B)$ for any $\alpha\in\bG_m$. For any positive real number $X$, let $S(F)_X$ denote the set of elements of $S(F)$ of height less than $X$ when $F$ is a number field and equal to $X$ when $F$ is a function field. Let $\Sigma_0$ be the fundamental domain for the action of $\bG_m(F)$ on $S(F)$ constructed in Proposition \ref{prop:Sigma}. Then $\Sigma_0$ is defined by congruence conditions and the intersection $\Sigma_0\cap S(F)_{X}$ is bounded. We shall view an elliptic curve over $F$ as an element of $\Sigma_0$.

A family of elliptic curves over $F$ defined by congruence conditions is a subset $\Sigma_1\subset \Sigma_0$ defined by congruence conditions. The characteristic function $\chi_{\Sigma_1}$ of $\Sigma_1$ factors as $\prod_\fP \chi_{\Sigma_1,\fP}$ and we denote by $\Sigma_{1,\fP}$ the subset of $S(F_\fP)$ whose characteristic function is $\chi_{\Sigma_1,\fP}$ for every $\fP$. 
A family given by $\Sigma_1$ is \emph{large} if for all but finitely many primes $\fP$, every element $(A,B)\in S(\cO_\fP)$ with $\Delta=4A^3-27B^2\notin\fP^2$ is contained in $\Sigma_{1,\fP}$. In this section, we prove the following generalization of Theorem \ref{ecthm}.

\begin{theorem}\label{ecthmlarge}
Let $n\in\{2,3,4,5\}$. Let $F$ be any global field of characteristic not dividing $6$ when $n\in\{2,3,4\}$ and characteristic not dividing $30$ when $n=5$.  Consider the set of all elliptic curves over $F$ in a large family $\Sigma_1$ ordered by height. Then the average size of the $n$-Selmer groups of $E$ is equal to $\sigma(n)$, the sum of divisors of $n$.
\end{theorem}

\subsection{Parametrization of $2$-, $3$-, $4$-, and $5$-Selmer group elements of elliptic curves}

For $n\in\{2,3,4,5\}$, we define the representations $(G_n,V_n)$ as follows:
\begin{table}[hbt]
\centering
\begin{tabular}{|c | c|c|}
\hline
$n$ & $G_n$ & $V_n$ \\
\hline
$2$ & $\PGL_2$ & $\Sym^4(2)$ \\[.05in]
$3$ & $\PGL_3$ & $\Sym^3(3)$  \\[.05in]
$4$ & $\{(g_2,g_4)\in\GL_2\times\GL_4:\det(g_2)\det(g_4)=1\}/\{(\lambda^{-2},\lambda)\}$ & $2\otimes\Sym^2(4)$ 
\\[.05in]
$5$ & $\{(g,g')\in\GL_5\times\GL_5:\det(g)^2\det(g')=1\}/\{(\lambda^{-2},\lambda)\}$ & $5\otimes\wedge^2(5)$ 
\\
\hline
\end{tabular}
\caption{Coregular representations parametrizing elements in $n$-Selmer groups of elliptic curves}\label{tablereps}
\end{table}

\noindent Here, the actions of $G_n$ on $V_n$ are:
\begin{equation*}
\begin{array}{rcl}
g\cdot f(x,y)&=&\displaystyle\frac{1}{\det(g)^2}f((x,y)\cdot g);\\[.15in]
g\cdot f(x,y,z)&=&\displaystyle\frac{1}{\det(g)}f((x,y,z)\cdot g);\\[.15in]
(g_2,g_4)\cdot (A,B)&=&(g_4Ag_4^t,g_4Bg_4^t)g_2^t;\\[.1in]
(g,g')(A,B,C,D,E)&=&(gAg^t, gBg^t, gBg^t, gDg^t, gEg^t)(g')^t.
\end{array}
\end{equation*}
It is well known that the representations $V_n$ of $G_n$ are coregular, with invariants freely generated by two polynomials $I_n$ and $J_n$. We single out another polynomial invariant $\Delta_n\in\Z[V]$ given by $\Delta_n=(4I_n^3-J_n^2)/27$, and call it the {\it discriminant}. Let $k$ be any field with characteristic not $2$ or $3$. Any element $v_n\in V_n(k)$ with invariants $I$ and $J$ and nonzero discriminant defines a smooth genus $1$ curve $C^{(v_m)}$ over $k$, naturally embedded in $\P^{n-1}$ when $n=3,4,5$ and in weighted projective space $\P(1,1,2)$ when $n=2$, whose Jacobian is the elliptic curve $E^{(v_n)}$ given by $y^2=x^3-\frac{I}3x-\frac{J}{27}$. More precisely: when $n=2$, the element $v_2(x,y)$ is a binary quartic form and $C^{(v_2)}$ is the curve cut out by the equation $z^2=v_2(x,y)$; when $n=3$, the element $v_3(x,y,z)$ is a ternary cubic form and $C^{(v_3)}$ is the curve cut out by $v_3(x,y,z)=0$; when $n=4$, the element $v_4=(A,B)$ defines two quadrics in $\P^3$, and $C^{(v_4)}$ is their intersection; when $n=5$, the element $v_5=(A_1,\ldots,A_5)$ is a $5$-tuple of alternating $5\times 5$ matrices, and $C^{(v_5)}$ is the curve cut out by their five Pfaffians. 
We say that $v_n$ is \emph{soluble} over $k$ if $C^{(v_n)}(k)$ is nonempty. For the global field $F$, we say that $v_n\in V_n(F)$ is {\it locally soluble} if it is soluble over $F_\fP$ for every place $\fP$ of $F$.

Then the following result follows from \cite[Theorems 4.1, 4.5, 4.11, 4.14]{BH} and \cite[Remarks 4.3, 4.9, 4.13, 4.15]{BH}:
\begin{theorem}
Assume that $F$ is a global field of characteristic not $2$, $3$ or $5$. Fix $n\in\{2,3,4,5\}$ and an elliptic curve $E=E_{A,B}$ over $F$. Then there exists a bijection between elements in the $n$-Selmer group of $E$ over $F$ and locally soluble $G_n(F)$-orbits on $V_n(F)$ having invariants $I_n(E)$ and $J_n(E)$, where $I_n(E)=\alpha_n A$ and $J_n(E)=\beta_n B$ for some fixed nonzero constants $\alpha_n,\beta_n\in\cO_F$.

Moreover, the stabilizer of any $v_n\in V_n(F)$ with nonzero discriminant is isomorphic to the finite flat group scheme $E^{(v_n)}[n]$.
\end{theorem}

Fix $n\in\{2,3,4,5\}$.
We impose the following condition of irreducibility on $V_n(F)$: an element $v_n\in V_n(F)$ is {\it generic} if it has nonzero discriminant and has order exactly $n$ in $\Sel_n(E^{(v_n)})$.
Let $\Sigma_1$ be a large family of elliptic curves. Then an application of Theorem \ref{thm:countingmain}, with appropriately defined weight functions (to ensure that the invariants of $v_n$ being counted align with $\Sigma_1$, and that local-solubility is enforced) will yield the average value of $|\Sel_n(E)|-1$ when $n\in\{2,3,5\}$ and $|\Sel_4(E)|-|\Sel_2(E)|$ when $n=4$. Therefore, to prove Theorem \ref{ecthmlarge}, it only remains to verify the various axioms upon which Theorem \ref{thm:countingmain} relies. 

AXIOM: \nicerep~is clear: all the groups in question are semisimple; the representations are all coregular; the sum of the degrees of the invariants equals $\dim V_n$ as $2+3=5$, $4+6=10$, $8+12=20$, and $20+30=50$ for $n=2$, $3$, $4$, and $5$, respectively; the generic stabilizers are finite. Conditions 1 and 2 for AXIOM: \localCondition~are immediately verified. Condition 3 is what is known as ``minimisation'', and are due to Cremona--Fisher--Stoll \cite{CFS} in the cases $n=2$, $3$, and $4$, and to Fisher \cite{Fisher} in the case $n=5$. See also work of Laga \cite{Laga}, giving a uniform proof for  this type of ``minimization'' result in many cases. Condition 4 follows from the same argument as in the proof of \cite[Lemma 30]{BS5}. 

AXIOM: Local Spreading follows from the existence of algebraic sections $\iota_n:S\to V_n$ defined over $\cO[1/6]$. When $n = 2$, we have $$\iota_2(I,J) = x^3y - \frac{I}{3}xy^3 - \frac{J}{27}y^4.$$ When $n = 3$, we have $$\iota_3(I,J) = x^3 - \frac{I}{3}xz^2 - \frac{J}{27}z^3 - y^2z.$$ When $n = 4$, we have
$$\iota_4(I,J) = 
 \left(
\left[ \begin{array}{cccc} 0 & 0 & 0 & 1\\ 0 & 0 & 0 & 0\\ 0 & 0 & 1 & 0\\ 1 & 0 & 0 & 0\end{array} \right],
\left[ \begin{array}{cccc} 0 & 0 & -1 & 0\\ 0 & -1 & 0 & 0\\ -1 & 0 & 0 & -I/6\\ 0 & 0 & -I/6 & -J/27\end{array} \right]
\right).
$$
When $n = 5$, the formula is written down in \cite[(4)]{BS5}. 
Composing with the quotient map $V(F_\infty)\to S(F_\infty)$ gives a homogeneous height function $H$ on $V(F_\infty)$. Since the number of $G(F_\infty)$-orbits on $V(F_\infty)$ having fixed degenerate invariants is absolutely bounded, AXIOM: Counting at Infinity I follows from AXIOM: Local Spreading as in Remark \ref{rem:RI}. 

AXIOM: Uniformity Estimate for $S$ is shown to be satisfied in the following result:
\begin{proposition}
Let $F$ be a global field of characteristic not $2$ or $3$, and denote $\Sigma_0\cap S(F)_X$ $($defined at the beginning of \S4$)$ by $\Sigma_{0,X}$. Then we have
\begin{equation*}
\sum_{N(\mathfrak{p})>M}\#\bigl\{(A,B)\in \Sigma_{0,X}:\mathfrak{p}^2\mid\Delta(A,B)\bigr\}\ll \frac{X^{10}}{M}+\frac{X^{10}}{\log X}.
\end{equation*}
\end{proposition}

\begin{proof}
    Suppose $\fP\nmid 6$. Let $W_\fP = \{(A,B)\in\Sigma_{0,X}\colon \fP^2\mid\Delta(A,B)\}$. Fix one of the possible $O(X^{4})$ possible choices for $A$ occurring as some $(A,B)\in\Sigma_{0,X}$ with $\fP\nmid A$. If $(A,B)\in W_\fP$ , then the reduction of $B$ modulo $\mathfrak{p}^2$ has at most an absolutely bounded number of choices. Hence the number of possible choices for $B$, once $A$ is fixed, is at most $O(X^{6}/N(\mathfrak{p}^2)+1)$ by Proposition \ref{prop:cong}. On the other hand, if $\fP\mid A$, then $(A,B)\in W_\fP$ would imply $\fP\mid B$. Hence the total number of $(A,B)\in W_\fP$ is $$\ll X^{4}\Big(\frac{X^{6}}{N\fP^2}+1\Big) + \Big(\frac{X^{4}}{N\fP}+1\Big)\Big(\frac{X^{6}}{N\fP}+1\Big) \ll \frac{X^{10}}{N\fP^2}+\frac{X^{6}}{N\fP}+X^{4}$$ Finally, note that every $(A,B)\in \Sigma_{0,X}$ satisfies $|\Delta(A,B)|\ll X^{12}$. Hence, if $\mathfrak{p}^2\mid\Delta(A,B)$, then $N(\mathfrak{p})\ll X^{6}$ and so there are at most $X^{6}/\log X$ possibilities for $\mathfrak{p}$. Summing over these prime ideals, we immediately obtain the required result.
\end{proof}

\begin{remark}
    Our height function defined in \eqref{eq:heightweightelliptic} is the $12$-th power of the height used in \cite{BS2}, which is why $\ord(H) = 10$ instead of $5/6$.
\end{remark}

Therefore, it only remains to check AXIOM: Counting at Infinity II, and AXIOM: Uniformity Estimate for $V$, which we do case by case in the next four subsections.

\subsection{$2$-Selmer: $\PGL_2$ acting on $\Sym^42$}

\subsubsection*{Verification of AXIOM: Counting at Infinity II:}

In this case, we have $G = \PGL_2$ and $V = \Sym^42$.
We verify AXIOM: Counting at Infinity II by checking the three conditions in Theorem \ref{prop:maintransference}. For Condition 1, note that an integral binary quartic form $f\in V(F)$ with nonzero discriminant corresponds to the trivial element in a $2$-Selmer group if it has a rational linear factor. Fix a finite prime $\fP$ and note that if a binary quartic form in $V(\cO)$ is non-generic, then its reduction modulo $\fP$ is non-generic in $V(k(\fP))$. The $\fP$-adic density of generic binary quadric forms in $V(k(\fP))$ is $1/4-\epsilon$ (the density of generic monic degree $4$ polynomials). Hence the $\fP$-adic density of non-generic binary quadric forms in $V(\cO)$ is at most $1-1/4+\epsilon$, thereby yielding Condition 1. 

Let $T$ denote the split torus consisting of diagonal elements $\diag(t^{-1},t).$ 
The representation $V$ decomposes as $\chi_{x^4}\oplus\chi_{x^3y}\oplus\chi_{x^2y^2}\oplus\chi_{xy^3}\oplus\chi_{y^4}$ when restricted to $T$, where
$$\chi_{x^4}(t) = t^{-4},\chi_{x^3y}(t)=t^{-2},\chi_{x^2y^2}(t)=1,\chi_{xy^3}(t)=t^2,\chi_{y^4}(t)=t^4.$$
Let $U_0$ denote the set consisting of these characters. A binary quartic form $f(x,y)$ is non-generic if its $x^4$ coefficient is $0$. We set $U_1=\{\chi_{x^4}\}$; then $f$ is generic only if it is $U_1$-generic. A basis $\Delta$ of the positive roots is given by the singleton $\{\alpha\}$ where $\alpha(t) = t^2$. The Haar measure character $\delta$ is given by $\delta(t) = t^{-2}$. Since $\chi_{x^4}$ is the minimal weight, we only need to check Condition 2 for $U=\emptyset$ which follows since $\delta=-\alpha\in\Q^{<0}[\Delta]$. Condition 3 is also immediate as $\alpha = -\chi_{x^3y}\in\Q^{\leq0}[U_0].$ Therefore, AXIOM: Counting at Infinity II is satisfied.

\subsubsection*{Verification of AXIOM: Uniformity Estimate for $V$:}

The verification of AXIOM: Uniformity Estimate for $V$ parallels the proof for $F=\Q$ in \cite[Theorem 2.13]{BS2}. The weight function $m_\fP$ on $V(F_\fP)$ for any $\fP\notin M_\infty$ was defined in \eqref{eq:defnofmp}. There are four possible reasons for $m_\fP(f)\neq 1$ for some $f\in V(\cO_\fP)$: the invariant $(I(f),J(f))$ of $f$ does not lie in $\Sigma_\fP$; $f$ is not $F_\fP$-soluble; the $F_\fP$-orbit $G(F_\fP)f$ breaks up into more than one $\cO_\fP$-orbit; the stabilizer of $f$ has more points over $F_\fP$ than over $\cO_\fP$. The proof of \cite[Proposition 3.18]{BS2} generalizes to arbitrary non-archimedean local fields implying that for all but finitely many $\fP$, we have $m_\fP(f)\neq 1$ only when the discriminant $\Delta(f)\in\fP^2$. We let $B_\fP$ be the subset of $V(\cO_\fP)$ consisting of binary quartic forms $f(x,y)$ such that $\Delta(f)\in\fP^2.$ It remains to verify AXIOM: Uniformity Estimate for the quadruple $(V, G,\{B_\fP\}, N).$ 

Let $B_\fP^{(1)}$ denote the set of $f\in V(\cO_\fP)$ such that $\Delta(f + \omega_\fP g)\in\fP^2$ for all $g\in V(\cO_\fP)$ where $\omega_\fP$ denotes a uniformizer of $\fP$. Write $B_\fP^{(2)}$ for the complement $B_{\fP}\backslash B_{\fP}^{(1)}$. AXIOM: Uniformity Estimate for the quadruple $(V, G,\{B_\fP^{(1)}\}, N)$ is proved via the geometric sieve in \cite[Theorem 2.18]{BS2}, and Theorem \ref{thm:Th7F} verifies it for general global fields. AXIOM: Uniformity Estimate for the quadruple $(V, G,\{B_\fP^{(2)}\}, N)$ is proved in \cite[Theorem 2.20]{BS2} by combining two different estimates, and we will verify it for arbitrary number fields $F$ in the same way. First, as there, we directly bound the $N(B_\fP^{(2)}, G, X)$ by fibering over the first four coefficients of an integral binary quartic form, and noting that the fifth coefficient is determined modulo $\fP^2$ up to an absolutely bounded number of choices. By Proposition \ref{prop:cong}, we have the bound
\begin{equation*}
N(B_\fP^{(2)}, G, X)\ll \frac{X^{10}}{N\fP^2}+X^{8}.
\end{equation*}
Summing the right hand side of the above estimate over primes $\fP$ with $M\leq N\fP\leq X^2$ gives an error
\begin{equation}\label{eq:unif_2Sel1}
\sum_{M\leq N\fP\leq X^2}N(B_\fP^{(2)}, G, X)\ll \frac{X^{10}}{M\log M}+\frac{X^{10}}{\log X}
\end{equation}
by the prime number theorem.

For our second step, we follow the strategy in \cite{BS2}, and embed $V$ in the larger space of pairs of ternary quadratic forms $W$, which admits an action of $\GL_2\times\GL_3$. In \cite{BS2}, we proceeded by using uniformity estimates on orbits in $W(\Z)$ from \cite{dodqf}, and bounding the fibers of the map from orbits in $V(\Z)$ to orbits in $W(\Z)$ by known uniform bounds on the number of solutions to the Thue equation. In our setting, we will use results from \cite{BSWfield} which prove uniformity estimates on orbits in $W(\cO)$. However, uniform bounds on the number of solutions to the Thue equation over arbitrary global fields do not seem to exist in the literature (at least to the authors' knowledge). Instead, we use an approach from \cite{SSW} of keeping track of the height of $f$ in terms of data from the $\GL_2(\cO)\times\SL_3(\cO)$-orbit. We use this extra information to bound the sizes of the fibers.

The embedding $V(\cO)\to W(\cO)$ is discriminant preserving, so the image of $B_\fP^{(2)}\cap V(\cO)$ lies inside the set of elements in $W(\cO)$ with norm discriminants $\ll X^{12}$ and divisible by~$\fP^2$. From \cite[(29)]{BSWfield}, the number of these orbits with nonzero discriminant is $\ll X^{12}/N\fP^2$. Fix one of these orbits $O$. Then this orbit corresponds to a quartic ring extension $Q$ of $\cO$ along with a cubic resolvent ring $R$ of $Q$ where $R$ is monogeneic over $\cO$. 
We know that the set of $\GL_2(\cO)$-orbits in $V(\cO)$ mapping to a fixed orbit in $W(\cO)$ is in bijection with the set of $\cO$-orbit of monogenizers $\alpha\in R$ (i.e. $R = \cO[\alpha]$), where $s\in\cO$ acts via $\alpha\mapsto\alpha+s$.  If $f$ is a $\GL_2(\cO)$-orbit in $V(\cO)$ corresponding to $\alpha$ such that the minimal polynomial of $\alpha$ over $\cO$ is $x^3 + ux^2 + Ix + J$ with $u$ belonging to a fixed set of representatives for $\cO/3\cO$, then we have
\begin{equation*}
I(f)\asymp I,\quad J(f)\asymp J.
\end{equation*}
For a place $v\in M_\infty$, recall that $\deg v$ denotes 
$[F_v:\R]$ when $F$ is a number field and $[F_v:\F_p(t)]$ when $F$ is a function field. We defined $d_F$ in \S\ref{sec:not} to be
\begin{equation*}
d_F=\sum_{v\in M_\infty} \deg v.
\end{equation*}

By our construction of $\Sigma_0$ in Proposition \ref{prop:Sigma}, we know that 
$$|I|_v^{1/\deg v} \ll X^{4/d_F}\qquad\mbox{and}\qquad |J|_v^{1/\deg v} \ll X^{6/d_F} \qquad\mbox{for all }v\in M_\infty.$$
Let $N_\infty$ denote the set of places of $\mathrm{Frac}(R)$ above $M_\infty$ and we write $\|\beta\| = \max_{\fQ\in N_\infty}|\beta|_\fQ^{1/\deg\fQ}$ for any $\beta\in R$. Then we have $\|\alpha\| \ll X^{2/d_F}$ for any $\alpha$ in $R$ with the above minimal polynomial. Let $\{1,\beta_1,\beta_2\}$ be a Minkowski basis for $R$ over $\cO$. Then viewing $R$ as a lattice $\L_R$ inside $\mathrm{Frac}(R)\otimes F_\infty$, we see that it has $d_F$ successive minima of size $\asymp 1$, $d_F$ successive minima of size $\asymp \|\beta_1\|$, and $d_F$ successive minima of size $\asymp \|\beta_2\|$. The discriminant of the lattice $\L_R$ equals the absolute discriminant $\Delta_R$ of $R$ up to a multiplicative constant so $$\|\beta_1\|\cdot\|\beta_2\| \asymp \Delta_R^{1/(2d_F)}.$$

Now suppose that the absolute discriminant of $R$ is $\gg X^{8+\delta}$ for some $\delta>0$ and suppose $R$ has a monogenizer $\alpha_0\notin\cO$ with $\|\alpha_0\| \ll X^{2/d_F}$. Then we have
$$\|\beta_1\| \ll X^{2/d_F}\qquad\mbox{and}\qquad \|\beta_2\| \gg X^{(2+\delta/2)/d_F}.$$
Then any $\alpha\in R$ with $\|\alpha\|\ll X^{2/d_F}$ is of the form $c_1\alpha_0 + c_2$ with $c_1,c_2\in\cO$. Such an $\alpha$ can be a monogenizer of $R$ only when $c_1\in\cO^\times$. We see that there can only be $O(1)$  possible such $\alpha$, since $|\alpha|_v^{1/\deg v}$ for $v\in M_\infty$ are bounded in terms of each other. So the contribution to $N(B_\fP^{(2)}, G, X)$ from elements $f$ with $|\Delta_2(f)|\gg X^{8+\delta}$ is
\begin{equation*}
\ll \frac{X^{12}}{N\fP^2}.
\end{equation*}
We next handle the contribution to $N(B_\fP^{(2)}, G, X)$ from elements $f$ with $|\Delta_2(f)|\asymp Y$, for $Y\ll X^{8+\delta}$. The number of orbits in $W(\cO)$ with this discriminant bound is now $\ll Y/N\fP^2$. Fix any one of these orbits (corresponding to $(Q,R)$). We count $\alpha\in R/\cO$ inside $(\mathrm{Frac}(R)\otimes F_\infty)/F_\infty$. Davenport's lemma (Proposition \ref{prop:Davenport}) gives the bound $X^4/\sqrt{Y} + X^{4 - 2/d_F}$ of these. Adding up over dyadic ranges of $Y$ yields the bound
\begin{equation*}
\frac{X^{8+\delta/2}}{N\fP^2}+\frac{X^{12+\delta-2/d_F}}{N\fP^2}.
\end{equation*}
Therefore, by taking $\delta<2/d_F$, we obtain
\begin{equation*}
N(B_\fP^{(2)}, G, X)\ll 
\frac{X^{12}}{N\fP^2}.
\end{equation*}

Summing this bounds over primes $\fP$ with $N\fP>X^2$, we obtain
\begin{equation}\label{eq:unif_2sel2t}
\sum_{N\fP> X^2}N(B_\fP^{(2)}, G, X)\ll 
\frac{X^{10}}{\log X}.
\end{equation}
Combining the two bounds \eqref{eq:unif_2Sel1} and \eqref{eq:unif_2sel2t} verifies the axiom.

\subsection{$3$-Selmer: $\PGL_3$ acting on $\Sym^33$}\label{sec:Sel3E}

\subsubsection*{Verification of AXIOM: Counting at Infinity II:}

In this case, we have $G=\PGL_3$ and $V=\Sym^3(3)$. A ternary cubic form in $V(F)$ with nonzero discriminant corresponds to the trivial element in a $3$-Selmer group if it contains a rational flex. If this happens, we say the form is \emph{non-generic}.
We verify AXIOM: Counting at Infinity II by checking the three conditions in Theorem \ref{prop:maintransference}. For Condition 1, \cite[Proof of Lemma 14]{BS3} gives $1-\lambda_\fP \gg (N\fP)^{-1}$, so the product of $\lambda_\fP$ diverges to $0$.

Let $T$ denote the split torus consisting of diagonal elements $t=\diag(s_1^{-2}s_2^{-1},s_1s_2^{-1},s_1s_2^2).$
A basis $\Delta$ of the positive roots is given by $\{\alpha_1,\alpha_2\}$ with $\alpha_1(t)=s_1^3,\alpha_2(t)=s_2^3.$ The representation $V$ decomposes as $\chi_{x^3}\oplus\chi_{x^2y}\oplus\chi_{x^2z}\oplus\cdots\oplus\chi_{z^3}$ when restricted to $T$, where
$$\chi_{x^3}=-2\alpha_1-\alpha_2,\,\chi_{x^2y}=-\alpha_1-\alpha_2,\,\chi_{x^2z}=-\alpha_1,\,\chi_{xy^2}=-\alpha_2,\,\chi_{xyz}=1,$$
$$\chi_{xz^2}=\alpha_2,\,\chi_{y^3}=\alpha_1-\alpha_2,\,\chi_{y^2z}=\alpha_1,\,\chi_{yz^2}=\alpha_1+\alpha_2,\,\chi_{z^3}=\alpha_1+2\alpha_2.$$
Let $U_0$ denote the set consisting of these characters. We set $U_1=\{\chi_{x^3},\chi_{x^2y},\chi_{x^2z}\}$ and $U_2=\{\chi_{x^3},\chi_{x^2y},\chi_{xy^2}\}$; then a ternary cubic form $f$ is non-generic if it is $U_1$-non-generic or $U_2$-non-generic. The Haar measure character $\delta$ is given by $\delta=-2\alpha_1-2\alpha_2$. Since $\chi_{x^2z}$ and $\chi_{xy^2}$ form a complete set of minimal characters in $U_0-\{\chi_{x^3},\chi_{x^2y}\}$, we only need to check Condition 2 for $U=\emptyset,$ for $U=\{\chi_{x^3}\}$ and for $U=\{\chi_{x^3},\chi_{x^2y}\}$. For $U=\emptyset$, we have $$\delta + \sum_{\chi\in U_0-U}\chi = -2\alpha_1-2\alpha_2\in\Q^{<0}[\Delta].$$ For $U=\{\chi_{x^3}\}$, we have $$\delta + \sum_{\chi\in U_0-U}\chi = -\alpha_2$$ and we can take $\pi_U = \epsilon\chi_{x^2y}$ for any real number $\epsilon$ satisfying $0<\epsilon<1$. For $U=\{\chi_{x^3},\chi_{x^2y}\}$, we have $$\delta + \sum_{\chi\in U_0-U}\chi = \alpha_1$$ and we can take $\pi_U = \epsilon_1\chi_{x^2z}+\epsilon_2\chi_{xy^2}$ for any real numbers $\epsilon_1,\epsilon_2$ satisfying $\epsilon_1>1,\epsilon_2>0,\epsilon_1+\epsilon_2 < 2$. Condition 3 is immediate as $\alpha_1 = -\chi_{x^2z}$ and $\alpha_2=-\chi_{xy^2}$ are in $\Q^{\leq0}[U_0].$ Therefore, AXIOM: Counting at Infinity II is satisfied.

\subsubsection*{Verification of AXIOM: Uniformity Estimate for $V$:}

The proof of AXIOM: Uniformity Estimate for $V$ for the case $F=\Q$ in \cite[Proposition 25]{BS3} generalizes to arbitrary global fields via the geometric sieve and the reduction sieve.

\subsection{$4$-Selmer: $G_4$ acting on $2\otimes\Sym^24$}\label{sec:Sel4E}

\subsubsection*{Verification of AXIOM: Counting at Infinity II:}

A pair $(A,B)$ of quaternary quadratic forms corresponds to a $2$-torsion element of a $4$-Selmer group exactly when its quartic resolvent form $\det(Ax+By)$ is non-generic. When this happens, we say that $(A,B)$ is {\it non-generic}. AXIOM: Counting at Infinity II is satisfied because Theorem \ref{prop:maintransference} is already proved in \cite{BS4}: Condition 1 is \cite[Lemma 16]{BS4}; Condition 2 is \cite[Lemma 15]{BS4}; Condition 3 is \cite[Lemma 14]{BS4}.

\subsubsection*{Verification of AXIOM: Uniformity Estimate for $V$:}
The verification of AXIOM: Uniformity Estimate for $V$ in \cite[Theorem 23]{BS4} generalizes to arbitrary global fields via the geometric sieve and the reduction sieve.

\subsection{$5$-Selmer: $G_5$ acting on $5\otimes\wedge^25$}\label{sec:Sel5E}

\subsubsection*{Verification of AXIOM: Counting at Infinity II:}
AXIOM: Counting at Infinity II is satisfied because Theorem \ref{prop:maintransference} is already proved in \cite{BS5}: Condition 1 is \cite[Proposition 22]{BS5}; Condition 2 is \cite[Proposition 18]{BS5}; Condition 3 follows as the the character associated to the $(1,2)$-coordinate of $A$ has the form $\sum_{\alpha\in\Delta} n_\alpha[\alpha]$ with all $n_\alpha<0$.

\subsubsection*{Verification of AXIOM: Uniformity Estimate for $V$:}

The verification of AXIOM: Uniformity Estimate for $V$ in \cite[Theorem 27]{BS5} generalizes to arbitrary global fields via the geometric sieve and the reduction sieve.

\subsection{Putting it all together}

With all necessary axioms verified for the representations $(G_n,V_n)$, it only remains to compute the local masses. We can do this in one stroke for each $n$, since we have
\begin{equation}\label{eq:mass}
c_\infty = \frac{\#E_{A,B}(F_\infty)/nE_{A,B}(F_\infty)}{\#E_{A,B}[n](F_\infty)},\quad c_\fP  = \frac{\#E_{A,B}(F_\fP)/nE_{A,B}(F_\fP)}{\#E_{A,B}[n](F_\fP)}.
\end{equation}
In each case, we have $c_\infty\prod c_\fP=1$. Therefore, applying Theorem \ref{thm:countingmain}, we see that for $n\in\{2,3,4,5\}$, the average size of the number of order $n$ elements in the $n$-Selmer groups of elliptic curves in $\Sigma_1$ is bounded by $\tau(G_n)=n$. Adding $1$ for the identity element for $n\in\{2,3,5\}$ and adding $1+2$ for elements of order dividing $2$ for $n=2$ yields Theorem \ref{ecthmlarge}, and hence also Theorem \ref{ecthm}.

Theorem \ref{thm:ellrank} now follows from Theorem \ref{ecthm} by noting that if the rank of an elliptic curve $E$ is $r$, then we have $|\Sel_5(E)|\geq 5^r\geq 20r-15$. Therefore the average rank of elliptic curves over $F$ (or indeed, in any large family of elliptic curves over $F$), ordered by height, is bounded by $21/20$.


\section{Rational points and ranks of Jacobians of hyperelliptic curves}\label{sec:hyper}
The goal of this section is to generalize the results in \cite{BG1}, \cite{SW}, and \cite{BGW} regarding the average sizes of the 2-Selmer groups of Jacobians of hyperelliptic curves.

\subsection{Monic hyperelliptic curves}
Fix a positive integer $m\geq1$ and let $F$ be a global field of characteristic not $2$. We consider monic degree $m$ hyperelliptic curves over $F$ given by an affine equation of the form
$$C_{c_1,\ldots,c_m}=C_f: y^2 = f(x)=x^{m} + c_1x^{m-1} + \cdots c_{m},$$
where $c_1,\ldots,c_{m}\in F$. When $m=2n+1$ is odd, then these are \emph{odd hyperelliptic curves} of genus $n$. When $m=2n+2$ is even, these are \emph{monic even hyperelliptic curves} of genus $n$.

Two curves $C_{c_1,\ldots,c_m}$ and $C_{c'_1,\ldots,c'_m}$ are equivalent if and only if there is some constant $\alpha\in F$ such that $c_i=\alpha^{2i}c'_i$ for all $i=1,\ldots,m$. Hence it is natural to view $(c_1,\ldots,c_m)$ as an element of the weighted projective space $\P(2,4,\ldots,2m).$ We define the height of $C_f$, and the height of $(c_1,\ldots,c_m)$, by the usual height on weighted projective space defined in Section \ref{sec:heightweighted}. That is, let $I$ be the ideal
$$I = \{a\in F:a^{2i}c_i\in \cO,\forall i=1,\ldots,m\},$$
then,
\begin{equation}\label{eq:heightweightell_hy}
H(c_1,\ldots,c_m) = (NI)\prod_{\fP\in M_\infty} \max(|c_1|_\fP^{1/2},\ldots,|c_m|_\fP^{1/(2m)}).
\end{equation}

Let $S=\A^{m}$ be the space of $n$-tuples $(c_1,\ldots,c_m)$ with an action of $\bG_m$ given by $\alpha.(c_1,\ldots,c_m)=(\alpha^2c_1,\ldots,\alpha^{2m}c_m)$ for any $\alpha\in\bG_m$. For any positive real number $X$, let $S(F)_X$ denote the set of elements of $S(F)$ of height less than $X$ when $F$ is a number field, and equal to $X$ when $F$ is a function field. Let $\Sigma_0$ be the fundamental domain for the action of $\bG_m(F)$ on $S(F)$ constructed in Proposition \ref{prop:Sigma}. Then $\Sigma_0$ is defined by congruence conditions and the intersection $\Sigma_0\cap S(F)_{X}$ is bounded. We shall view a monic degree $m$ hyperelliptic curve over $F$ as an element of $\Sigma_0$.

A family of monic degree $m$ hyperelliptic curves over $F$ defined by congruence conditions is a subset $\Sigma_1\subset \Sigma_0$ defined by congruence conditions. We say a family given by $\Sigma_1$ is \emph{large} if for all but finitely many primes $\fP$, every element $(c_1,\ldots,c_m)\in S(\cO_\fP)$ with $\Delta(x^m + c_1x^{m-1} + \cdots + c_m)\notin\fP^2$ is contained in $\Sigma_{1,\fP}$. We prove the following generalization of Theorem \ref{hyperecthm}.

\begin{theorem}\label{thm:oddhyper}
Fix a positive integer $m$ and a global field $F$ of characteristic not $2$. When all monic degree $m$ hyperelliptic curves over $F$ in a large family $\Sigma_1$ are ordered by height, the average size of the $2$-Selmer groups of their Jacobians is bounded above by~$3$ if $m$ is odd, and $6$ if $m$ is even.
\end{theorem}

\subsubsection{Odd hyperelliptic curves}\label{sec:hyperelliptic}

We now specialize to the case where $m=2n+1$ is odd. Let $G=\SO_{2n+1}$ be the special orthogonal group of the split quadratic form $Q_0$ on a $2n+1$ dimensional vector space acting on the space $V$ of self-adjoint operators by conjugation. The ring of polynomial invariants is freely generated by the coefficients $c_1,\ldots,c_{2n+1}$ on the characteristic polynomial of $T$ defined by
$$f(x)=\det(xI - T) = x^{2n+1} + c_1x^{2n} + \cdots + c_{2n+1}.$$
The degree of $c_i$ is $i$ for all $i=1,\ldots,2n+1$. This verifies AXIOM: \nicerep. Fix any family $\Sigma_1$ of odd hyperelliptic curves over $F$ defined by finitely many congruence conditions. Let $\kappa_1$ be a nonzero element of $\cO$ such that $\kappa_1.\Sigma_1\subset S(\cO)$. 

Let $k$ be any field of characteristic not $2$. We say a self-adjoint operator $T\in V(k)$ is \emph{stable} if it is regular semisimple, or equivalently if its characteristic polynomial has no repeated factors. Let $T\in V(k)$ be stable with characteristic polynomial $f(x)$. Then its stabilizer scheme is isomorphic to $J_f[2]$ (\cite[Proposition 11]{BG1}). Moreover, there is an associated torsor $F_T[2]$ of $J_T[2]$ which allows one to embed the set of $G(k)$-orbit in $V_f(k)$ into $H^1(k, J_f[2])$. The \emph{distinguished} orbit maps to the identity element and the \emph{soluble} orbits map to the Kummer embedding of $J_f(k)/2J_f(k)$ in $H^1(k,J_f[2])$ (\cite[Proposition 12]{BG1}). We say a stable element $T\in V(k)$ is \emph{generic} if its orbit is not distinguished. For a global field $F$, the \emph{locally soluble} orbits, consisting of stable $T\in V(F)$ that is $F_\fP$-soluble for every prime $\fP$ of $F$, are in bijection with the $2$-Selmer group $\Sel_2(J_f/F).$

We have the following generalization of \cite[Proposition 19]{BG1} with the same proof.

\begin{proposition}\label{prop:oddhypersol}
Let $\kappa\in\cO$ be a fixed nonzero element so that every element in $1+\kappa^4\cO_\fP$ is a square for any a non-archimedean place $\fP$ of $F$. Fix any non-archimedean place $\fP$ of $F$. Suppose $T\in V(F_\fP)$ is $F_\fP$-soluble and its invariant lies in $\kappa.S(\cO_\fP)$. Then $T$ is $G(F_\fP)$-conjugate to an element in $V(\cO_\fP)$.
\end{proposition}

Set $\Sigma=\kappa^2\kappa_1.\Sigma_1$. Let $V_{\Sigma}(F)$ be the set of locally soluble $T\in V(F)$ with invariants in $\Sigma$ and for any prime $\fP$, let $V_{\Sigma,\fP}(F)$ be the set of soluble stable  $T\in V(F_\fP)$ with invariants in $\Sigma_{\fP}$. Let $m_\Sigma$ be the characteristic function of $V_{\Sigma}(F)$ and for any prime $\fP$, let $m_{\Sigma,\fP}$ be the characteristic function of $V_{\Sigma,\fP}(F)$. Then $m_\Sigma=\prod_\fP m_{\Sigma,\fP}$ is a local product and each $m_{\Sigma,\fP}$ is $G(F_\fP)$-invariant. Then the weight function $m_\Sigma$ satisfies AXIOM: \localCondition: Condition 3 is Proposition \ref{prop:oddhypersol}; Condition 4 is in the proof of \cite[Proposition 32]{BG1}.

The quotient map $V\rightarrow S$ admits and algebraic section $S\rightarrow V$ defined over $\cO[1/2]$ (\cite[\S4.1]{BG1}. This verifies AXIOM: Local Spreading.

When restricted to $\Sigma$, the height function defined in \eqref{eq:heightweightell_hy} is given by
$$H(c_1,\ldots,c_{2n+1}) = (N(\kappa^2\kappa_1))^{-1}\prod_{\fP\in M_\infty} \max(|c_1|_\fP^{1/2},\ldots,|c_{2n+1}|_\fP^{1/(4g+2)}).$$
We extend this naturally to $S(F_\infty)$. Composing with the quotient map $V(F_\infty)\rightarrow S(F_\infty)$ gives a height function $H$ on $V(F_\infty)$ homogeneous of degree $1/2.$ Since the number of $F_\infty$-orbits is absolutely bounded, Remark \ref{rem:RI} gives the pre-compact fundamental domains $R_\lambda$. Hence AXIOM: Counting at Infinity I is satisfied.

AXIOM: Counting at Infinity II is satisfied because Theorem \ref{prop:maintransference} is already proved in \cite{BG1}: Condition 1 is \cite[Proposition 31]{BG1}; Condition 2 is \cite[Proposition 29]{BG1}; Condition 3 is shown in the proof of \cite[Proposition 29]{BG1}. AXIOM: Uniformity Estimate for $S$ is proved in \cite[Theorem 1.2]{Oller}.

The various masses are exactly like \eqref{eq:mass}, where
$$c_\infty = \frac{\#J_f(F_\infty)/2J_f(F_\infty)}{\#J_f[2](F_\infty)},\quad c_\fP  = \frac{\#J_f(F_\fP)/2J_f(F_\fP)}{\#J_f[2](F_\fP)},$$
do not depend on $f\in \Sigma$. The product formula for abelian varieties gives $c_\infty\prod_{\fP\notin M_\infty}c_\fP  = 1.$ Therefore, by Theorem \ref{thm:countingmain}, the average number of generic locally soluble orbits with invariant in $\Sigma$ is bounded above by $\tau_{G,F}=2.$

\subsubsection{Monic even hyperelliptic curves}\label{sec:monicevenhyper}

Suppose now $m=2n+2$ is even. There are now two points $\infty,\infty'$ at infinity. AXIOMS: \nicerep, \localCondition, Counting at Infinity I, Local Spreading and Uniformity Estimate for $S$ are exactly like the case of $n$ odd in the previous section using the results of \cite{SW}. Let $k$ be any field of characteristic not $2$. The distinguished orbits now correspond to the identity element and the image of $(\infty)'-(\infty)$ in $H^1(k,J_f[2])$. For a global field $F$, they are distinct 100\% of the time, by \cite[Proposition 30]{SW}, when ordered by $H(f)$. The analogous Proposition \ref{prop:oddhypersol} follows from \cite[Proposition 12]{SW}. 

AXIOM: Counting at Infinity II is checked in \cite{SW} over $\Q$ but unlike \cite{BG1}, it was not proved by checking the three conditions of Theorem \ref{prop:maintransference}. Condition 1 is \cite[Proposition 23]{SW}. We now check Conditions 2 and 3.

Let $G=\PSO_{2n+2}$ be the special orthogonal group of the split quadratic form $Q_0$ on $\Z^{2n+2}$ modulo the center $\mu_2$. It acts on the space $V$ of self-adjoint operators by conjugation. We fix a basis $\{e_1,\ldots,e_{2n+2}\}$ so that the Gram matrix of $Q_0$ is given by the matrix $A$ with $1$'s on the anti-diagonal and $0$'s elsewhere. We identify the space of self-adjoint operators $T$ with the space of symmetric matrices, which we also denote by $V$, under the map $T\mapsto AT$. The action of $g\in G$ on some $B\in V$ is given by $g.B = gBg^t.$ We use $b_{ij}$ to denote the $(i,j)$-entry of some element $B$ of $V$. Let $T$ be the split torus of $G$ consisting of diagonal elements $t=\diag(t_1^{-1},\ldots,t_{n+1}^{-1},t_{n+1},\ldots,t_1)$. A basis $\Delta$ of positive roots is given by the set $\{\alpha_1,\ldots,\alpha_{n+1}\}$, where
\begin{eqnarray*}
    \alpha_i(t)&=&t_i/t_{i+1},\mbox{ for }i=1,\ldots,n,\\
    \alpha_{n+1}(t)&=&t_nt_{n+1}.
\end{eqnarray*}
Denote by $\chi_{b_{ij}}$ the character associated to the $b_{ij}$-coordinate. For example, we have
\begin{eqnarray}
 \label{eq:Cond3prop}\chi_{b_{11}} &=& -2\alpha_1 - \cdots - 2\alpha_{n-1} - \alpha_n - \alpha_{n+1},\\
 \nonumber \chi_{b_{i\,2n+2-i}} &=& -\alpha_i,\quad i=1,\ldots,n,\\
 \nonumber \chi_{b_{n+1\,n+1}} &=& \alpha_n - \alpha_{n+1}.
\end{eqnarray}
Let $U_0=U(n)=\{\chi_{b_{ij}}\}_{i\leq j}$ denote the set of these characters. Note \eqref{eq:Cond3prop} implies that $\alpha_i\leq -\chi_{b_{11}}$ for all $i=1,\ldots,n+1$. This verifies Condition 3 of Theorem \ref{prop:maintransference}. The Haar measure  $\delta=\delta(n)$ is given by (\cite[(20)]{SW}):
\begin{equation}\label{eq:delta}
 \delta(n) = \sum_{j=1}^{n-1} j(j-2n-1)\alpha_j - \frac{n(n+1)}{2}(\alpha_n + \alpha_{n+1}).
\end{equation}
The following proposition gives a criterion for an element $B\in V(F)$ to be non-generic.

\begin{proposition}\label{prop:reduciblecondition}
 Suppose $B\in V(F)$ such that either
\begin{enumerate}
 \item $b_{ij}=0$ for all pairs $i\in I$, $j\in J$ where $I,J$ are two subsets of $\{1,\ldots,2n+2\}$ satisfying $|I|+|J|\geq2n+2$, or
 \item $b_{ij}=0$ for all pairs $i,j$ satisfying $i<n$ and $i+j<2n+2$, and $b_{n, n} = b_{n,n+1} = 0$, or
 \item $b_{ij}=0$ for all pairs $i,j$ satisfying $i<n$ and $i+j<2n+2$, and $b_{n, n} = b_{n,n+2} = 0$.
\end{enumerate}
Then $B$ is non-generic.
\end{proposition}

\begin{proof}
 In the first case, the discriminant of the characteristic polynomial $f(x)=\det(xA - B)$ is $0$ and hence $B$ is not stable (\cite[Lemma 27]{BG1}). In the second and third case, if $T=A^{-1}B$ is stable, then $T$ lies in the distinguished orbits. Here distinguished means that there is an $n$-dimensional subspace $W$ isotropic with respect to $A$ and $B$, and an $(n+1)$-dimensional subspace $W'$ containing $W$ isotropic with respect to $A$ such that $W\perp W'$ with respect to $B$ (\cite[Proposition 2.31]{W2}). We may take
 $$W=\Span\{e_1,\ldots,e_n\}$$
 and $W'$ to be $\Span\{e_1,\ldots,e_n,e_{n+1}\}$ in Case 2 and $\Span\{e_1,\ldots,e_n,e_{n+2}\}$ in Case 3.
\end{proof}

Let $U(n)_{I,J}$ denote the set of coordinates appearing in Case 1 of Proposition \ref{prop:reduciblecondition} for subsets $I,J$ of $\{1,\ldots,2n+2\}$ satisfying $|I|+|J|\geq2n+2$, let $U(n)_0$ denote the set of coordinates appearing in Case 2, and let $U(n)_1$ denote the set of coordinates appearing in Case 3. Let $U$ be a saturated subset of $U(n)$ that does not contain $U(n)_0$, $U(n)_1$ or any of the $U(n)_{I,J}$. We need to find $\pi_U\in\Q^{\geq0}[U(n)-U]$ with $\deg(\pi_U) < \#U$ such that
\begin{equation}\label{eq:Cond3}
 \pi_U + \delta(n) + \sum_{\chi\in U(n)-U}\chi\in\Q^{<0}[\Delta].
\end{equation}
When $U=\emptyset$, we can take $\pi_\emptyset = 0$ since $\delta(n)\in \Q^{<0}[\Delta]$ by \eqref{eq:delta}. Suppose from now on $U$ is not empty. For any $k=1,\ldots,n+1$, denote by $e_k(U)$ and $\delta(n)_k$ the coefficients of $\alpha_k$ in $\sum_{\chi\in U(n)-U}\chi$ and $\delta(n)$, respectively, when viewed inside $\Q[\Delta]$. For any such $k$, we define
\begin{eqnarray*}
 U^k &=& \{\chi_{b_{ij}}\in U:i=k\},\\
 U^{\geq k} &=& \{\chi_{b_{ij}}\in U: i\geq k\}.
\end{eqnarray*}
For any $k=1,\ldots,n-1,$ let $\pi_k$ be a minimal element of the set $U(n)^k-U^k$. We remark that there is a unique choice of $\pi_k$ unless $U^k=\{\chi_{b_{kk}},\ldots,\chi_{b_{kn}}\}$, in which case there are two choices and it doesn't matter which one is picked to be $\pi_k$. The following table gives the definitions of $\pi_n$ and $\pi_{n+1}$.

\begin{table}[ht]
  \centering
  \begin{tabular}{|c | c| c| }
    \hline
    $U^n$ & $\pi_n=-\alpha_n$ & $\pi_{n+1}=-\alpha_{n+1}$ \\
    \hline
    \hline
    $\subset\{\chi_{b_{nn}}\}$ & $\chi_{b_{n,n+2}}$ & $\chi_{b_{n,n+1}}$ \\
\hline
$\{\chi_{b_{nn}},\chi_{b_{n,n+1}}\}$ & $\chi_{b_{n,n+2}}$ & $\chi_{b_{n+1,n+1}}+\chi_{b_{n,n+2}}$\\
\hline
$\{\chi_{b_{nn}},\chi_{b_{n,n+2}}\}$ & $\chi_{b_{n,n+1}}+\chi_{b_{n+2,n+2}}$ & $\chi_{b_{n,n+1}}$ \\
\hline
  \end{tabular}
  \caption{Definitions of $\pi_n$ and $\pi_{n+1}$}
\label{table1}
\end{table}

If $U^n$ contains any more than the cases in the above table, then $U$ will contain $U(n)_{I,J}$ for some subsets $I,J$ of $\{1,\ldots,2n+2\}$ satisfying $|I|+|J|\geq2n+2$. These $\pi_k$'s are chosen so that they are in $\Q[U(n)-U]$ and such that $\pi_k\leq-\alpha_k$ for $k=1,\ldots,n+1$ with equality achieved when $k=n,n+1$. 
Set $$\pi_U = \sum_{k=1}^{n+1}(\epsilon + \max(0,e_k(U)+\delta(n)_k))\pi_k,$$
for some small positive constant $\epsilon$ to be chosen. Since each $\pi_k\leq-\alpha_k$, we see that \eqref{eq:Cond3} follows immediately. To check $\deg(\pi_U)<\#U,$ it suffices to show that
\begin{equation}\label{eq:Cond3deg}
 \sum_{k=1}^{n+1}\max(0,e_k(U)+\delta(n)_k)\deg(\pi_k) < \#U.
\end{equation}

We prove \eqref{eq:Cond3deg} by induction on $n$. Suppose first $n=1$. Then $U$ can only be $\{b_{11}\}$ and we have $$\delta(1) = -\alpha_1-\alpha_2,\qquad\mbox{and}\qquad\sum_{\chi\in U(1)-U}\chi = -\chi_{b_{11}} = \alpha_1+\alpha_2$$ so the statement is clear. Suppose $n>1$. The reduction from $n$ to $n-1$ is by removing the first and last rows and columns. The basis $\Delta$ of positive roots becomes $\{\alpha_2,\ldots,\alpha_n\}$. If $U^{\geq2}$ contains some $U(n-1)_{I,J}$, then $U$ contains $U(n)_{\{1\}\cup I, \{1\}\cup J}$. If $U^{\geq2}$ contains $U(n-1)_0$ or $U(n-1)_1$, then they are equal and $U=U(n)_0-\{\chi_{b_{1\,2n}}\}$ or $U = U(n)_1-\{\chi_{b_{1\,2n}}\}$ since $U$ is assumed to be saturated. In this case, \eqref{eq:Cond3deg} follows from a direct computation. Suppose now $U^{\geq2}$ does not contain $U(n-1)_{I,J}$ for every $I,J\subset\{2,\ldots,2n+2\}$ with $|I|+|J|\geq 2n$ or $U(n-1)_0$ or $U(n-1)_1$. By induction, we then have
\begin{equation}\label{eq:indhyp}
 \sum_{k=2}^{n+1}\max(0,e_k(U^{\geq2})+\delta(n-1)_k)\deg(\pi_k) < \#U^{\geq2}.
\end{equation}
From \eqref{eq:delta}, we have
$$\delta(n) = \delta(n-1) - 2n\sum_{j=1}^{n-1}\alpha_j - n(\alpha_n + \alpha_{n+1}).$$
There are three different possibilities for $U^1$:
\begin{description}
 \item[Case 1:] $U^1=\{\chi_{b_{11}},\ldots,\chi_{b_{1r}}\}$ with $r\leq n$. Direct computation shows $$e_k(U^1) + \delta(n)_k-\delta(n-1)_k< 0,$$ for any $k=1,\ldots,n+1$. Combining with \eqref{eq:indhyp} gives \eqref{eq:Cond3deg}.
 \item[Case 2:] $U^1=\{\chi_{b_{11}},\ldots,\chi_{b_{1n}},\chi_{b_{1l}}\}$ with $l=n+1$ or $n+2$. Direct computation shows $$\sum_{k=1}^{n+1}(e_k(U^1) + \delta(n)_k-\delta(n-1)_k))\pi_k = \sum_{k=1}^{n-1}(k-n+1)\pi_k + \pi_{2n+2-l}.$$
 Hence $$\sum_{k=1}^{n+1}\max(0,e_k(U)+\delta(n)_k)\deg(\pi_k)<\#U^{\geq2} + 2 \leq\#U.$$
 \item[Case 3:] $U^1=\{\chi_{b_{11}},\ldots,\chi_{b_{1r}}\}$ with $n+2\leq r\leq 2n$. Direct computation shows $$\sum_{k=1}^{n+1}(e_k(U^1) + \delta(n)_k-\delta(n-1)_k))\pi_k = \sum_{k=1}^{\max(n-1,2n+2-r)}(r+k-2n)\pi_k + \sum_{k=2n+3-r}^{n-1}2\pi_k + \pi_n + \pi_{n+1}.$$
Hence
\begin{eqnarray*}
 \sum_{k=1}^{n+1}\max(0,e_k(U)+\delta(n)_k)\deg(\pi_k)&<&\#U^{\geq2} + 1 + 2 + 2(r-n-3) + 1 + 2 \\
 [-.15in]&=& \#U^{\geq2} + 2r - 2n \\&\leq& \#U^{\geq2} + r \\&=& \#U.
\end{eqnarray*}
\end{description}
We have now completed the verification of Condition 2 of Theorem \ref{prop:maintransference}. As a result, AXIOM: Counting at Infinity II holds.

The various masses are exactly like the monic odd case with the only difference being that $\tau_{G,F}=4.$

\subsection{General hyperelliptic curves}

Fix a genus $n\geq 1$ and let $F$ be a global field of characteristic not $2$. We consider the family of general (even degree) hyperelliptic curves, as curves in weighted projective space $\P(1,1,n+1)$ and expressed by an equation of the form
\begin{equation}\label{hypereq1}
C_{c_0,\ldots,c_{2n+2}}=C_f: z^2 = f(x,y) = c_0x^{2n+2}+c_1x^{2n+1}y+\cdots+c_{2n+2}y^{2n+2},
\end{equation}
where $c_0,\ldots,c_{2n+2}\in F$. Two curves $C_{c_0,\ldots,c_{2n+2}}$ and $C_{c'_0,\ldots,c'_{2n+2}}$ are (defined to be) equivalent if and only if there is some constant $\alpha\in F$ such that $c_i=\alpha^{2}c'_i$ for all $i=0,\ldots,2n+2$. Hence it is natural to view $(c_0,\ldots,c_{2n+2})$ as an element of the weighted projective space $\P(2,2,\ldots,2).$ We define the height of $C_f$, and the height of $(c_0,\ldots,c_{2n+1})$, by the usual height on weighted projective space defined in Proposition \ref{prop:Sigma}. That is, let $I$ be the ideal
$$I = \{a\in F:a^{2}c_i\in \cO,\forall i=2,\ldots,2n+2\}.$$
Then,
\begin{equation}\label{eq:heightweightellHyper}
H(c_0,\ldots,c_{2n+2}) = (NI)\prod_{\fP\in M_\infty} \max(|c_0|_\fP^{1/2},\ldots,|c_{2n+2}|_\fP^{1/2}).
\end{equation}

Let $S=\A^{2n+3}$ be the space of $2n+3$-tuples $(c_0,\ldots,c_{2n+2})$ with an action of $\bG_m$ given by $\alpha.(c_0,\ldots,c_{2n+2})=(\alpha^2c_0,\ldots,\alpha^{2}c_{2n+2})$ for any $\alpha\in\bG_m$. For any positive real number $X$, let $S(F)_X$ denote the set of elements of $S(F)$ of height less than $X$ when $F$ is a number field, and equal to $X$ when $F$ is a function field. Let $\Sigma_0$ be the fundamental domain for the action of $\bG_m(F)$ on $S(F)$ constructed in Section \ref{sec:heightweighted}. Then $\Sigma_0$ is defined by congruence conditions and the intersection $\Sigma_0\cap S(F)_{X}$ is bounded. We shall view a monic degree $m$ hyperelliptic curve over $F$ as an element of $\Sigma_0$.

A \emph{large} family of hyperelliptic curves over $F$ is a subset $\Sigma_1\subset \Sigma_0$ defined by congruence conditions such that there exists subscheme $S_0$ of $S$ of codimension at least $2$ such that for all but finitely many non-archimedean places $\fP$, we have $\Sigma_0(k(\fP))\backslash \Sigma_1(k(\fP)) \subset S_0(k(\fP))$; where $\Sigma_0(k(\fP))$ and $\Sigma_1(k(\fP))$ denote the reduction modulo $\fP$ of the $\fP$-adic completion of $\Sigma_0$ and $\Sigma_1$ in $S(F_\fP)$ respectively. The set of locally soluble hyperelliptic curves over $F$ in $\Sigma_0$ is a large family (\cite[Lemma 15]{PS}). In this section we focus on large families of locally soluble hyperelliptic curves over $F$ and prove the following generalization of Theorem \ref{thm:GHsel2count}.

\begin{theorem}\label{thm:GHsel2countgen}
 Fix an integer $n\geq1$ and a global field $F$ of characteristic not $2$. Let $\Sigma_1$ be a large family of hyperelliptic curves over $F$ all of which are locally soluble. Then when hyperelliptic curves $C$ in $\Sigma_1$ are ordered by height, the average size of $\Sel_2(J^1)$ is bounded above by $2$.
\end{theorem}

\begin{remark}
   The requirement for large families here is stricter than the cases for monic hyperelliptic curves. This is because the results of \cite{BSWsqf2} have not been generalized to global fields, so we do not have the necessary squarefree sieves. With the current weaker notion of large families, AXIOM: Uniformity Estimate for $S$ follows from the geometric sieve (Theorem \ref{thm:geosieve}).
\end{remark}

The group $G=\SL_{2n+2}/\mu_2$ acts on the space $V=\Sym_2(2n+2)\oplus\Sym_2(2n+2)$ of pairs of symmetric bilinear forms via $g.(A,B)=(g^tAg,g^tBg).$ The ring of polynomial invariants is freely generated by the coefficients of its \emph{invariant binary form} $$f(x,y)=(-1)^{n+1}\det(Ax - By) = c_0x^{2n+2}+c_1x^{2n+1}y+\cdots+c_{2n+2}y^{2n+2}.$$
There are $2n+3$ of them each of degree $2n+2$. This verifies AXIOM: \nicerep. 

For any field $k$ of characteristic not $2$, an element $(A,B)\in V(k)$ is \emph{stable} if its invariant binary form $f(x,y)$ has no repeated factors. A stable element $(A,B)\in V(k)$ has stabilizer scheme isomorphic to $J_f[2]$ and gives a two-cover $F_{A,B}\to J_f^1$ of $J_f^1 = \Pic^1(C_f)$ (\cite[Theorem 2.7]{W2}). We say $(A,B)$ is $k$-\emph{soluble} if $F_{A,B}(k)\neq\emptyset$. The set of soluble orbits with invariant binary form $f(x,y)$ is in bijection with $J^1_f(k)/2J_f(k)$, which is acted on simply-transitively by the $2$-group $J_f(k)/2J_f(k)$. 

An element of $V(F)$ is \emph{generic} if and only if it is stable. A stable element $(A,B)\in V(F)$ is \emph{locally soluble} if $(A,B)$ is $F_\fP$-soluble for every place $\fP$ of $F$. If $f(x,y)$ is a binary $(2n+2)$-ic form whose associated hyperelliptic curve $C_f;z^2=f(x,y)$ is locally soluble, then there is a bijection between the set of locally soluble orbits with invariant binary form $f(x,y)$ and the $2$-Selmer set $\Sel_2(J^1)$ (\cite[Theorem 33]{BGW}).  

We have the following generalization of \cite[Theorem 32]{BGW} with the same proof.

\begin{proposition}\label{prop:evenhypersol}
Let $\kappa\in\cO$ be a fixed nonzero element so that every element in $1+\kappa^4\cO_\fP$ is a square for any a non-archimedean place $\fP$ of $F$. Fix any non-archimedean place $\fP$ of $F$. Suppose $(A,B)\in V(F_\fP)$ is $F_\fP$-soluble and its invariant lies in $\kappa^2.S(\cO_\fP)$. Then $T$ is $G(F_\fP)$-conjugate to an element in $V(\cO_\fP)$.
\end{proposition}

Let $\kappa_1$ be a nonzero element of $\cO$ such that $\kappa_1.\Sigma_1\subset S(\cO)$. Set $\Sigma=\kappa^2\kappa_1.\Sigma_1$. Let $V_{\Sigma}(F)$ be the set of locally soluble $T\in V(F)$ with invariants in $\Sigma$ and for any prime $\fP$, let $V_{\Sigma,\fP}(F)$ be the set of soluble stable  $T\in V(F_\fP)$ with invariants in $\Sigma_{\fP}$. Let $m_\Sigma$ be the characteristic function of $V_{\Sigma}(F)$ and for any prime $\fP$, let $m_{\Sigma,\fP}$ be the characteristic function of $V_{\Sigma,\fP}(F)$. Then $m_\Sigma=\prod_\fP m_{\Sigma,\fP}$ is a local product and each $m_{\Sigma,\fP}$ is $G(F_\fP)$-invariant. Then the weight function $m_\Sigma$ satisfies AXIOM: \localCondition, where Condition 4 was proved in the proof of \cite[Proposition 13]{B}.

We now check AXIOM: Local Spreading. For all the cases we have treated in the previous sections, the projection maps $V\rightarrow S$ all had algebraic sections defined over $\cO[1/N]$ for some $N\in\cO$. In this case, we no longer have such a section! Fix any place $\fP$ of $F$ and any $v=(A,B)\in V(F_\fP)$ with $m_\fP(v)\neq0.$ Then by the definitions of $F$ and $m_\fP$, $(A,B)$ is $F_\fP$-soluble and the associated hyperelliptic curve $C_f:z^2=f(x,y)$ is also $F_\fP$-soluble where $f(x,y)$ is the invariant binary form of $(A,B)$. Let $Q_f$ be a $F_\fP$-rational non-Weierstrass point of $C_f$. Then there exists a $\fP$-adically open neighborhood $R$ of $f$ in $S(F_\fP)$ such that $C_{f'}$ is $F_\fP$-soluble for any $f'\in R$ and the rational point $Q_f$ moves continuously (analytically) over $R$. That is, over $R$, we have a family of hyperelliptic curves with a marked rational non-Weierstrass point. Therefore, we can use the algebraic section in \S\ref{sec:monicevenhyper} to obtain the desired local section $R\rightarrow V(F_\fP)$.

When restricted to $F$, the height function defined in \eqref{eq:heightweightellHyper} is given by
$$H(c_0,\ldots,c_{2n+2}) = (N(\kappa^2\kappa_1))^{-1}\prod_{\fP\in M_\infty} \max(|c_0|_\fP^{1/2},\ldots,|c_{2n+2}|_\fP^{1/2}).$$
We extend this naturally to $S(F_\infty)$. Composing with the quotient map $V(F_\infty)\rightarrow S(F_\infty)$ gives a height function $H$ on $V(F_\infty)$ homogeneous of degree $n+1.$ Since the number of $F_\infty$-orbits is absolutely bounded, Remark \ref{rem:RI} gives the pre-compact fundamental domains $R_\lambda$. Hence AXIOM: Counting at Infinity I is satisfied.

AXIOM: Counting at Infinity II is satisfied because the conditions of Theorem \ref{prop:maintransference} are already proved in \cite{B}: Condition 1 is \cite[Proposition 13]{B}; Condition 2 is \cite[Proposition 12]{B}; Condition 3 follows as the the character associated to the $(1,1)$-coordinate of $A$ has the form $\sum_{\alpha\in\Delta} n_\alpha[\alpha]$ with all $n_\alpha<0$.

The various masses are exactly like \eqref{eq:mass}, where
$$c_\infty = \frac{\#J_f(F_\infty)/2J_f(F_\infty)}{\#J_f[2](F_\infty)},\quad c_\fP  = \frac{\#J_f(F_\fP)/2J_f(F_\fP)}{\#J_f[2](F_\fP)},$$
do not depend on $f\in \Sigma$. The product formula for abelian varieties gives $c_\infty\prod_{\fP\notin M_\infty}c_\fP  = 1.$ Therefore, by Theorem \ref{thm:countingmain}, the average number of (generic) locally soluble orbits with invariant in $\Sigma$ is bounded above by $\tau_{G,F}=2.$

\vspace{10pt}
\noindent\textbf{Proof of Theorem \ref{thm:generalhyperelliptic}}: Since the order of magnitude of the number of (generic) integral orbits is the same as the order of magnitude of the number of hyperelliptic curves, the statement on the rarity of rational points follows by the same argument as in \cite{B}. The statement on odd degree points follows from Theorem \ref{thm:GHsel2count} just as \cite[Theorem 3]{BGW} follows from \cite[Theorem 4]{BGW} using results of Dokchitser and Dokchitser (\cite[Appendix A]{BGW}).\hfill$\Box$

\subsection*{Acknowledgments}
It is a pleasure to thank Calvin Deng, Wei Ho, Marti Oller, Ari Shnidman, and Artane Siad for helpful conversations and comments on an earlier version of this article. MB was supported by a Simons Investigator Grant and NSF grant DMS-1001828. AS and XW were supported by NSERC discovery grants.


\begin{thebibliography}{10}


\bibitem{B}
M.\ Bhargava, Most hyperelliptic over $\Q$ curves have no rational points, \url{http://arxiv.org/abs/1308.0395}.

\bibitem{dodqf} M.\ Bhargava, The density of discriminants of quartic
  rings and fields, {\it Ann.\ of Math.\ $($2$)$} {\bf 162} (2005), no.\ 2, 1031--1063.

\bibitem{dodpf} M.\ Bhargava, The density of discriminants of quintic
  rings and fields, {\it Ann.\ of Math.\ $($2$)$} {\bf 172} (2010), no.\ 3,  1559--1591.

\bibitem{Bgeosieve}
M.\ Bhargava,
The geometric sieve and the density of squarefree values of invariant polynomials,
\url{http://arxiv.org/abs/1402.0031}.


\bibitem{BG1}
M.\ Bhargava and B.\ Gross,
The average size of the $2$-Selmer group of the Jacobians of hyperelliptic curves with
a rational Weierstrass point, {\it Automorphic Representations and $L$-functions}, {\it TIFR Studies
in Math.} {\bf 22} (2013), 23--91.

\bibitem{BGW} M.\ Bhargava, B.\ Gross. and X.\ Wang, A positive
  proportion of locally soluble hyperelliptic curves over $\Q$ have no
  point over any odd degree extension, {\it J. Amer. Math. Soc.} {\bf
    30} (2017), 451--493.


\bibitem{BH} M.\ Bhargava and W.\ Ho, Coregular spaces and genus one curves, {\it Camb.\ J.\ Math.} {\bf 4} (2016), no.\ 1, 1--119.

\bibitem{BH1} M.\ Bhargava and W.\ Ho, On average sizes of Selmer groups and ranks in families of elliptic curves having marked points, \url{https://arxiv.org/abs/2207.03309} 

\bibitem{BS2} M.\ Bhargava and A.\ Shankar, Binary quartic forms
  having bounded invariants, and the boundedness of the average rank
  of elliptic curves, {\it Ann.\ of Math.\ $($2$)$}, {\bf 181} (2015), no.\ 1,  191--242.


\bibitem{BS3}
M.\ Bhargava and A.\ Shankar, Ternary cubic forms having bounded invariants, and the existence of a positive proportion of elliptic curves having rank $0$, {\it Ann.\ of Math.\ $($2$)$} {\bf 181} (2015), no.\ 2,  587--621.

\bibitem{BS4}
M.\ Bhargava and A.\ Shankar, The average number of elements in the 4-Selmer groups of elliptic curves is 7, \url{http://arxiv.org/abs/1312.7333}.

\bibitem{BS5}
M.\ Bhargava and A.\ Shankar, The average number of elements in the 5-Selmer groups of elliptic curves is 6, and the average rank is less than 1, \url{http://arxiv.org/abs/1312.7859}.




\bibitem{BSWfield}
M.\ Bhargava, A.\ Shankar and X.\ Wang, Geometry-of-numbers methods over global fields I: Prehomogeneous vector spaces, \url{http://arxiv.org/abs/1512.03035}.

\bibitem{BSWsqf1}
M.\ Bhargava, A.\ Shankar and X.\ Wang, Squarefree values of polynomial discriminants I, {\it Invent.\ Math.} {\bf 228} (2022) no.\ 3, 1037--1073.

\bibitem{BSWsqf2}
M.\ Bhargava, A.\ Shankar and X.\ Wang, Squarefree values of polynomial discriminants II, {\it Forum Math.\ Pi} {\bf 13} (2025), e17.

\bibitem{rankone}
M.\ Bhargava and C.\ Skinner, A positive proportion of elliptic curves over $\Q$ have rank one, {\it J.\ Ramanujan Math.\ Soc.\ } {\bf 29} (2014), no.\ 2,  221--242. 

\bibitem{bsddensity}
M.\ Bhargava, C.\ Skinner, and W.\ Zhang, A majority of elliptic curves over $\Q$ satisfy the Birch and Swinnerton-Dyer Conjecture, \url{http://arxiv.org/abs/1407.1826}.

\bibitem{BV2}
M.\ Bhargava and I.\ Varma, On the mean number of 2-torsion elements in the class groups, narrow class groups, and ideal groups of cubic orders and fields, {Duke Math.\ J.} {\bf 10} (2015), 1911--1933.

\bibitem{BV}
M.\ Bhargava and I.\ Varma, The mean number of 3-torsion elements in the class groups and ideal groups of quadratic orders, {\it Proc.\ Lond.\ Math.\ Soc. $($3$)$} {\bf 112} (2016), no.\ 2,  235--266








\bibitem{Borel}
A.\ Borel, Some finiteness properties of adele groups over number fields, {\it Publ. Math. IHES} {\bf 16} (1983), 5--30.

\bibitem{BHC}
A.\ Borel and Harish-Chandra, Arithmetic subgroups of algebraic groups, {\it Ann.\ of Math.} {\bf 75} (1962), 485--535

\bibitem{BP}
A.\ Borel, G.\ Prasad, Finiteness thoerems for discrete subgroups of bounded covolume in semi-simple groups, {\it Publ. Math. IHES} {\bf 69} (1989), 119--171.

\bibitem{Conrad}
B.\ Conrad, Finiteness theorems for algebraic groups over function fields, {\it Compos. Math.} {\bf 148} (2012), no.\ 2, 555--639.

\bibitem{CFS}
J.\ E.\ Cremona, T.\ A.\ Fisher and M.\ Stoll, Minimalisation and reduction of 2-, 3-, and 4-coverings of elliptic curves, {\it Algebr.\ Number Theory} {\bf 4} No.\ 6 (2010), 763--820.

\bibitem{DH}
H.\ Davenport and H.\ Heilbronn, On the density of discriminats of cubic fields II, {\it Proc. Roy. Soc. London. Ser. A} \textbf{322} (1971), no.\ 1551, 405--420.

\bibitem{dJ}
A.\ J.\ de Jong, Counting elliptic surfaces over finite fields, {\it Mosc. Math. J.} {\bf 2} (2002), no.\ 2, 281-–311.

\bibitem{AWD}
A.\ W.\ Deng, Rational points on weighted projective spaces, \url{ http://arxiv.org/abs/math/9812082}.



\bibitem{Fisher}
T.\ Fisher, Explicit 5-descent on elliptic curves, {\it ANTS X--Proceedings of the Tenth Algorithmic Number Theorem Symposium}, 395--411.


\bibitem{Go}
D.\ Goldfeld, Conjectures on elliptic curves over quadratic fields, {\it Number Theory, Carbondale 1979 (Proc. Southern Illinois Conf.)}, 108--118. Lecture Notes in Math. {\bf 751}, Springer, Berlin, 1979.



\bibitem{Laga} J.\ Laga, Graded Lie algebras, compactified Jacobians and arithmetic statistics, {\it J.\ Eur.\ Math.\ Soc.} (2024).


\bibitem{LR} J.\ Laga, The average size of the 2-Selmer group of a family of non-hyperelliptic curves of genus 3, {\it Algebr.\ Number Theory} {\bf 16}, no. 5 (2022), 1161--1212.

\bibitem{Ngo}
Q.\ P.\ Ho, V.\ B.\ Le Hung and B.\ C.\ Ngo, Average size of 2-Selmer groups of elliptic curves over function fields, \url{http://arxiv.org/abs/1310.7963}.

\bibitem{KS}
N.\ M.\ Katz and P.\ Sarnak, Random matrices, Frobenius eigenvalues, and monodromy, {\it American Mathematical
Society Colloquium Publications} {\bf 45}, American Mathematical Society, Providence, RI, 1999.

\bibitem{Oller}
M.\ Oller, Geometry-of-numbers over number fields and the density of ADE families of curves having squarefree discriminant,
\url{https://arxiv.org/abs/2505.11301}.

\bibitem{PS2}
B.\ Poonen, M.\ Stoll, Most odd degree hyperelliptic curves have only one rational point, {\it Ann.\ of Math. $(2)$} {\bf 180} (2014), 1137--1166.

\bibitem{PS}
 B.\ Poonen and M.\ Stoll, The Cassels--Tate pairing on polarized
 abelian varieties, {\it Ann.\ of Math.} {\bf 150} (1999), 1109--1149.

\bibitem{RT} B.\ Romano, and J.\ Thorne, $E_8$ and the average size of the 3‐Selmer group of the Jacobian of a pointed genus‐2 curve, {\it Proc.\ of the London Math.\ Soc.\,} {\bf 122 (5)}, (2021), 678--723.

\bibitem{ANS}
A.\ N.\ Shankar, $2$-Selmer groups of hyperelliptic curves with marked points, {\it Trans. Amer. Math. Soc.} {\bf 372} (2019), 267--304.

\bibitem{SSW}
A.\ N.\ Shankar, A.\ Shankar, and X.\ Wang, Large families of elliptic curves ordered by conductor,
{\it Compos.\ Math.\ }{\bf 157} (2021), no. 7, 1538--1583.

\bibitem{ST}
A.\ Shankar, The average rank of elliptic curves over number fields, Ph.D thesis, Princeton 2012.


\bibitem{SW}
A.\ Shankar and X.\ Wang, Rational points on hyperelliptic curves having a
marked non-Weierstrass point, {\it Compos.\ Math.} {\bf 154} (2018), 188--222.

\bibitem{Siad1}
A.\ Siad, Effect of monogenicity on 2-torsion in the class group of number fields of odd degree,
\url{https://arxiv.org/abs/2011.08834}.

\bibitem{Siad2}
A.\ Siad, Monogenic fields with odd class number Part II: even degree,
\url{https://arxiv.org/abs/2011.08842}.

\bibitem{Swaminathan1}
A.\ A.\ Swaminathan, Most odd-degree binary forms fail to primitively represent a square, {\it Compos. Math.}, {\bf 160(3)} (2024), 481--517. 

\bibitem{Sp}
T.\ A.\ Springer. Reduction theory over global fields, {\it Proc.\ Indian Acad.\ Sci.\ (Math.\ Sci.\ )} {\bf 104} (1994), 207--216.


\bibitem{Th}
J.\ Thorne, $E_6$ and the arithmetic of a family of non-hyperelliptic curves of genus 3, {\it Forum Math.\ Pi}, Vol.\ 3 (2015), e1.

\bibitem{Th1}
J.\ Thorne, On the average number of 2-Selmer elements of elliptic curves over $\F_q(X)$ with two marked points.
{\it Doc.\ Math.} {\bf 24} (2019), pp. 1179--1223.

\bibitem{W2}
X.\ Wang, {\it Maximal linear spaces contained in the base loci
of pencils of quadrics}, {\it Algebr.\ Geom.} {\bf 5} (2018), no.\ 3, 359--397.

\bibitem{Weil} A.\ Weil, {\it Adeles and algebraic groups}, Birkh\"auser, 1982.

\end{thebibliography}
\end{document}